\documentclass{article}

\usepackage{amsmath,amsthm,amsfonts,amssymb}
\usepackage{bbold}

\usepackage{geometry}
\usepackage{ifpdf,graphicx,subfig}
\ifpdf
  \DeclareGraphicsExtensions{.eps,.pdf,.png,.jpg}
\else
  \DeclareGraphicsExtensions{.eps}
\fi

\usepackage{color}
\usepackage[colorlinks=true]{hyperref}

\usepackage{mdframed}

\usepackage{booktabs}
\usepackage{algorithm,algorithmic}
\usepackage[sort&compress,numbers]{natbib}
\newcommand{\Rmnum}[1]{\expandafter\@slowromancap\romannumeral #1@}
\renewcommand{\labelenumi}{(\roman{enumi})}

\usepackage{siunitx}
\theoremstyle{plain}
\newtheorem{theorem}{Theorem}[section]
\newtheorem{lemma}[theorem]{Lemma}
\newtheorem{proposition}[theorem]{Proposition}

\theoremstyle{definition}
\newtheorem{definition}{Definition}[section]

\theoremstyle{remark}
\newtheorem*{remark}{Remark}

\numberwithin{equation}{section}

\DeclareMathOperator{\sgn}{sgn}
\DeclareMathOperator{\supp}{supp}
\DeclareMathOperator{\dom}{dom}
\DeclareMathOperator{\dist}{dist}

\newcommand{\bfn}{\mathbf{n}}
\newcommand{\bfs}{\mathbf{s}}

\newcommand{\bfu}{\mathbf{u}}
\newcommand{\bfv}{\mathbf{v}}

\newcommand{\bfx}{\mathbf{x}}
\newcommand{\bfy}{\mathbf{y}}
\newcommand{\bfz}{\mathbf{z}}
\newcommand{\bfh}{\mathbf{h}}

\newcommand{\bfA}{\mathbf{A}}
\newcommand{\bfB}{\mathbf{B}}
\newcommand{\bfI}{\mathbf{I}}

\newcommand{\bbR}{\mathbb{R}}
\newcommand{\calE}{\mathcal{E}}
\newcommand{\calF}{\mathcal{F}}
\newcommand{\sfi}{\mathsf{i}}
\newcommand{\sfg}{\mathsf{g}}
\newcommand{\sfG}{\mathsf{G}}
\newcommand{\sfS}{\mathsf{S}}

\newcommand{\sfm}{\mathsf{m}}

\newcommand{\sfN}{\mathsf{N}}

\newcommand{\ve}{\varepsilon}

\newcommand{\set}[1]{\left\{#1\right\}}
\newcommand{\abs}[1]{\left\vert#1\right\vert}
\newcommand{\norm}[1]{\left\Vert#1\right\Vert}
\renewcommand{\vec}[1]{\mbox{\boldmath \small $#1$}}

\allowdisplaybreaks

\renewcommand{\abstract}[1]{\textbf{Abstract.}~#1}
\providecommand{\keywords}[1]{\textbf{Keywords.}~#1}
\providecommand{\subjclass}[2]{\textbf{Mathematics subject classification (2010).}~#1}

\begin{document}


\title{\bf An Efficient and Globally Convergent Algorithm for $\ell_{p,q}$-$\ell_{r}$ Model  in Group Sparse Optimization}
\author{Yunhua Xue,Yanfei Feng and
	Chunlin Wu\thanks{Corresponding author. wucl@nankai.edu.cn},
	\\
{\em School of Mathematical Sciences and LPMC, Nankai University, Tianjin {\rm 300071}, China}}

\date{\today}

\maketitle

\abstract{ Group sparsity combines the underlying sparsity and group structure of the data in problems. We develop a proximally linearized algorithm InISSAPL for the non-Lipschitz group sparse $\ell_{p,q}$-$\ell_r$ optimization problem. The algorithm gives a unified framework for all the parameters $p\geq 1,0<q<1,1\leq r\leq \infty$, which is applicable to different kinds of measurement noise. In particular, it includes the addition of the non-smooth $\ell_{1,q}$ regularization term and the non-smooth $\ell_1$/$\ell_{\infty}$ fidelity term as special cases. It allows an inexact inner loop accessible to the implementation of scaled ADMM, and still has global convergence. The algorithm is efficient and fast with computation only on the shrinking group support set. Many numerical experiments are presented for the algorithm with diversity of parameters $p,q,r$. The comparisons show that our algorithm is superior to others in the existing works.

}



\section{Introduction}
\label{sec:Intro}

We consider the following $\ell_{p,q}$-$\ell_{r}$ minimization problem
\begin{equation}\label{eq:lp-mini-prob}
	\min_{\bfx \in \bbR^{N}}
    \calE (\bfx):= \norm{\bfx}_{p,q}^{q} + F_r(\bfx),
\end{equation}
where
\begin{equation*}
	F_r(\bfx)=\left\{\begin{array}{ll}
\displaystyle\frac{1}{r\alpha} \norm{\bfA\bfx-\bfy}_r^r,&r\geq 1,\\
\displaystyle\frac{1}{\alpha} \norm{\bfA\bfx-\bfy}_{\infty},&r=\infty,
\end{array}\right.
\end{equation*}
and $ p \in[1,\infty)$, $q \in (0,1) $,$r\in [1,\infty]$, $ \alpha \in( 0, \infty)$, $\bfA \in \bbR^{M  \times N}$, $\bfx \in \bbR^{N} $, $\bfy \in \bbR^{M}$, the $\ell_{p,q}$ regularization term measures the group sparse structure of $\bfx$, which is a quasi-norm, defined by
$$\norm{\bfx}_{p,q}= \left(\sum_{\sfi=\sf1}^{\sfg}\norm{\bfx_{\sfi}}_{p}^{q}\right)^{1/q},$$
where $\bfx_{\sfi},\sfi=\sf1,\cdots,\sfg$ are the group members defined in Section \ref{sec:Pre} and $\|\cdot\|_p$ is the standard $L^p$ norm for vectors.

In Big Data era, data used to describe the structures, segments and features always have group property. Namely, they have a natural grouping of their components. Sparsity allows us to reconstruct high-dimensional data with only a small number of variables, leading to better recovery performance. By combining them, the recovery or reconstruction of group sparse data is enhanced to an active research topic in sparse optimization. The group sparse minimization problem~\eqref{eq:lp-mini-prob} by underdetermined linear measurements has a wide variety of applications, such as signal recovery~\cite{Deng2013GSO, HuYaohuaetal2017GSO}, image processing~\cite{q-wang2017}, compressed sensing~\cite{Usman2011MRI}, model selection in birth weight prediction~\cite{Yuan2006Group}, sparse learning~\cite{Yang2010ICML}, variable selection in gene finding~\cite{Meier2008JRSS} and so on. Therefore, it is meaningful to study efficient algorithms for this general group sparse optimization problem.

The \emph{general} means that it covers a lot of case models for different parameters $p,q,r$. We assume the observation
\[\bfy = \bfA \bfx + \bfn,\]
where $\bfn \in \bbR^{M}$ represents the noise. The model here can be adapted for the diversity of noise by the parameter $r$ in the data fitting term $F_r(\bfx)$. As well known, for Gaussian noise, people use the $\ell_{2}$ fidelity term ($r=2$). For Laplace noise or heavy-tailed noise such as impulsive noise, the $\ell_{1}$ fidelity term ($r=1$) is a good choice. For the noise by uniform distribution or quantization error, the $\ell_{\infty}$ fidelity term ($r=\infty$) suits.


There are many references to study the sparse optimization problem without group structure in it, i.e. the non-group model in which the number of groups $\sfg$ equals $N$. Then the $\ell_{p,q}$ term in ~\eqref{eq:lp-mini-prob} is degenerated to $\ell_q~(0<q<1)$ regularization one. One class of methods is smoothing approximate methods~\cite{Bian-Chen2013,Chen2010Lower,Chen2013Optimality,Lanza2015Generalized,Chan2014Half}. By a smoothing function $\varphi(x,\theta)$, the non-Lipschitz property of the objective function can be removed. The second class of methods is general iterative shrinkage-thresholding algorithms~(GISA) for $\ell_{q}$-$\ell_{2}$ problem~\cite{Xu2012L12,Zuo2013Generalized,Bredies2015Minimization}. GISA was inspired by the great success of soft thresholding and iterative shrinkage-thresholding algorithms (ISTA)~\cite{Daubechies2004iterative,Beck2009fast} for convex $\ell_{1}$-$\ell_{2}$ problem. The third class of methods is the iterative reweighted minimization methods for $\ell_{q}$-$\ell_{2}$ minimization problem; see, e.g.~\cite{Lai2013Improved,Candes2008Enhancing,Chen2014Convergence,Lu2014}. Actually reweighted methods reformulate the original non-Lipschitz $\ell_{q}$-$\ell_{2}$ to Lipschitz ones by a de-singularizing parameter. Very recently, \cite{Zhifangliuaxiv,zeng2018} developed methods by successively shrinking the support of the variables to overcome non-Lipschitz property, in which \cite{Zhifangliuaxiv} considered the non-group case with $r\neq \infty$ and \cite{zeng2018} focused on the image restoration with $r=2$. To the best of our knowledge, we note that most of the references considered only $r=2$ in these methods.

For the group sparse optimization problem \eqref{eq:lp-mini-prob}, most algorithms were proposed only in the case of $r= 2$ as well. Hu et al. \cite{HuYaohuaetal2017GSO} investigated this problem via $\ell_{p,q}$ regularization, others developed algorithms for $\ell_{2,1}$ regularized least squares, e.g. group Lasso \cite{Yuan2006Group, Boyd2011Distributed,Deng2013GSO}. As noted before, it is important and necessary to develop algorithm for general $1\leq r\leq \infty$. This will bring the difficulty to universally handle the noise parameter $r$ with regularized parameters $p,q$ in the group structure. In addition, the regularization term $\ell_{p,q}$ with parameters $p\geq1, 0<q<1$ in the objective function $\calE$ in~\eqref{eq:lp-mini-prob} leads to a non-convex, non-Lipschitz optimization problem. This non-smoothness becomes even serious for the $\ell_{1,q}$ regularization case. All these characteristics of the minimization model \eqref{eq:lp-mini-prob} result in a great challenge to solve it.

In this paper, we extend our recent work \cite{Zhifangliuaxiv} to solve the general group sparse optimization problem (\ref{eq:lp-mini-prob}). This extension is not trivial, because model (\ref{eq:lp-mini-prob}) is more complicated and includes more nonsmooth cases than the non-group one in \cite{Zhifangliuaxiv}, as mentioned in the former paragraph.
We firstly establish a motivating proposition by developing subdifferential lemmas in group variables. This gives us the rationality to design a unified iterative support shrinking algorithm over group support set of unknown variables for various $p,q,r$.
To make the algorithm more practical and easily implementable, we linearize the regularization term and present the InISSAPL algorithm to calculate the approximate solution. Although the algorithm allows an inexact inner loop, we prove its global convergence from a new lower bound theory for the $\ell_p$ norm of the nonzero groups of iteration sequence. The algorithm implementation by scaled
ADMM is also discussed where, especially for the case of $r=\infty$, we give an analytical derivation of the explicit
solution of the corresponding subproblem. Numerical experiments show that the algorithm is not only robust to the diversity of noise, but also has good performance for different $p,q$. Compared with others in group sparse optimization on relative errors, successful rates and running time, our algorithm outperforms them. The main
characters of InISSAPL algorithm for model (\ref{eq:lp-mini-prob}) are presented as follows,

\begin{enumerate}
\item The algorithm provides a unified framework for all the parameters $p,q,r$. It can particularly deal with the case of the addition of non-smooth $\ell_{1,q}$ regularization term and non-smooth $\ell_{1}$/$\ell_{\infty}$ fidelity term.
\item The computation is implemented only on the shrinking group support set of $\bfx$ at each iteration step. Naturally our algorithm is efficient, especially for large scale sparse recovery problems.
\item  The key step is to overcome the non-Lipschitz property of the objective function and construct an appropriate subdifferential formula, when using KL property to prove the global convergence of the algorithm. It is solved by developing a lower bound theory of the nonzero groups of the iterative sequence and a technical construction of the subdifferential; see section 4 for details.
\end{enumerate}

The rest of the paper is outlined as follows. In section~\ref{sec:Pre}, we give some basic notations and preliminaries. In section~\ref{sec:algo}, we give the motivating proposition and propose the corresponding algorithms. In section~\ref{sec:con-ana}, we establish the global convergence theorem for the proposed algorithms. In section~\ref{sec:imlpem}, we describe the implementation of the algorithm by scaled ADMM. Numerical experiments and comparisons are showed in section~\ref{sec:experiments}. Section~\ref{sec:conclusion} concludes the paper.

\section{Notations and preliminaries}
\label{sec:Pre}
Suppose that $\bfA$ is an $M\times N$ matrix and $\bfx$ is a column vector with $N$ components. $I = \set{1,2,\dots, M }$ denotes the row index set of $\bfA$. To be specialized, we use another kind of upright font to express the group index such as $\sfG,\sfi,\sfg$. Let $\bfx:=\left(\bfx_{\sf1}^T,\bfx_{\sf2}^T,\cdots,\bfx_{\sfg}^T\right)^T$ represent the group structure of $\bfx$. $\sfG =\set{\sf1,\sf2,\dots,\sfg}$ denotes the group index set of $\bfx$. For each group member $\bfx_{\sfi}$, we denote by $J_{\sfi}=\{1,2,\cdots,N_{\sfi}\}$ the index set, then $N=N_{\sf1}+\cdots + N_{\sfg}$.
We also refer to $\bfx_{\sfi,j}$ as its $j$th entry of $\bfx_{\sfi}$ and denote the group support set of $\bfx$ by
\[
\supp_{\sfG}(\bfx):=\set{\sfi \in \sfG: \bfx_{\sfi} \neq \bf0},
\]
where $\bfx_{\sfi}\neq \bf0$ means that $\bfx_{\sfi,j}\neq 0$ for some $j\in J_{\sfi}$. Furthermore, we use $\bfx_{\sfi}=\bf0$ when $\bfx_{\sfi,j}=0$ for all $j\in J_{\sfi}$. The support of group member $\bfx_{\sfi}$ is defined by
\[
\supp(\bfx_{\sfi})=\set{j\in J_{\sfi}: \bfx_{\sfi,j}\neq 0}.
\]

Let $\sfS$ be a subset of $\sfG$.
We denote by $\bfx_{\sfS}$ the group vectors of $\bfx$  indexed by $\sfS$, which consists of the nonzero group members of $\bfx$ when $\sfS = \supp_{\sfG}(\bfx)$.

For a matrix $\bfA \in \bbR^{M \times N}$, we partition it into submatrices $A_{k,\sfi}, k\in I, \sfi\in \sfG$, which is the $k$th row of $\bfA$ partitioned according to the group structure of $\bfx$, i.e.,
\[
\bfA =
\left[
\begin{array}{cccc}
A_{1,\sf1}&A_{1,\sf2}&\cdots &A_{1,\sfg}\\
  \cdots&\cdots&\cdots&\cdots \\
  A_{M,\sf1}&A_{M,\sf2}&\cdots&A_{M,\sfg}\\
\end{array}
\right].
\]
Because $A_{k,\sfi}, k\in I, \sfi\in \sfG$ are row vectors, we denote by $(A_{k,\sfi})_j$ the $j$-th entry of it.
In a similar way with $\bfx_{\sfS}$, we denote by $\bfA_{\sfS}$ the column sub-matrix of $\bfA$ consisting of the columns indexed by $\sfS$.

Define $\phi: [0,\infty) \to [0,\infty) $ by $\phi(x) = x^{q} (0 < q < 1)$. We state some useful properties for $\phi(\cdot)$.
\begin{proposition}\label{prop-phi}
The function $\phi(\cdot)$ has the following properties:
\begin{enumerate}
\item $\phi(0) = 0$ and $\phi^{\prime}(x) = q x^{q-1}  > 0$ on $(0,\infty)$.
\item  $\phi(x)$ is concave and the following inequality holds,
\begin{equation}\label{eq:first-approx}
	\phi(y) \leq \phi(x) + \phi^{\prime}(x)(y - x),\; \forall x\in (0,\infty), y \in [0,\infty).
\end{equation}
\item For any $c > 0$, $\phi^{\prime}(x)$ is $ L_{c}$-Lipschitz continuous on $[c,\infty)$, i.e., there exists a constant $ L_{c} > 0$ determined by $c$, such that $\forall x,y \in [c, \infty)$,
\begin{equation}\label{eq:grad-Lip-cond}
	\abs{\phi^{\prime}(x) - \phi^{\prime}(y)} \leq L_{c} \abs{x-y}.
\end{equation}
\end{enumerate}
\end{proposition}

\begin{lemma}\label{lemma:norm}Let $\bfy\in \mathbb{R}^m$ be the m-dimensional vector, the following inequality holds:
$$\norm{\bfy}_{\gamma_2}\leq \norm{\bfy}_{\gamma_1}, 0< \gamma_1\leq \gamma_2.$$

\end{lemma}
\begin{proof}
Let $f(t)=\|\bfy\|_t, ~t>0$, then $f(t)$ is monotone decreasing by the fact $f^{\prime}(t)<0$ for $t>0$.

\end{proof}

\begin{lemma}\label{lemma:norm-equivalence}
Let $s>0, \bfy\in \mathbb{R}^m$, then there exists constant $C_s>0$, such that,
\begin{equation*}
\norm{\bfy}_s\leq C_s\norm{\bfy}_{s+1}.
\end{equation*}
\end{lemma}

\begin{proof}
For $s\geq 1$, the result can be verified easily from the norm equivalence in finite dimensional space.  For $0<s<1$, from \cite[Lemma 1]{HuYaohuaetal2017GSO}, we have
\begin{equation*}
\norm{\bfy}_s\leq m^{1-2^{-Z}} \norm{\bfy}_2,
\end{equation*}
where $Z$ is the smallest integer such that $2^{Z-1}s\geq 1$. We use the norm equivalence once again to have
\begin{equation*}
\norm{\bfy}_s\leq C_s \norm{\bfy}_{s+1},
\end{equation*}
where $C_s=m^{1-2^{-Z}}\cdot C$, and $C$ is the relation coefficient of norm equivalence.
\end{proof}

\section{Motivation and the proposed algorithm}
\label{sec:algo}

\subsection{Subdifferentials and regularity}

By the definition of $\phi(\cdot)$, we have
$\norm{\bfx}_{p,q}^{q}=\sum_{\sfi \in \sfG}\phi(\norm{\bfx_{\sfi}}_p)$. We also define the norm function $g(\bfy)=\norm{\bfy}_p$ for a vector $\bfy$. In order to calculate the subdifferential of the object function $\mathcal{E}(\bfx)$ in \eqref{eq:lp-mini-prob}, we give two lemmas firstly.

\begin{lemma}[Subdifferential]\label{lem-subdifferential}
Let $\bfy\in\mathbb{R}^m$ be an $m$-dimensional vector, we have the following results,
\begin{enumerate}
\item For $\bfy=\bf0$ and $p\geq 1$, the subdifferential is,
$$\partial (\phi \circ g)(\bfy)=\prod_{j=1}^m S_j,$$
where $S_j=(-\infty,\infty),\forall j=1,2,\cdots,m$ and $\Pi$ means the Cartesian product of sets;
\item For $\bfy\neq \bf0$, the subdifferential would be
$$\partial (\phi \circ g)(\bfy)=\prod_{j=1}^m S_j,$$
where
\[
     S_j =
       \begin{cases}
         \phi^{\prime}(\norm{\bfy}_p)\norm{\bfy}_p^{1-p}|\bfy_{j}|^{p-1}\sgn(\bfy_{j}), & p> 1,\\
         \phi^{\prime}(\norm{\bfy}_1)\sgn(\bfy_{j}),& j\in \supp(\bfy) \mbox{ and } p=1,\\
         [-\phi^{\prime}(\norm{\bfy}_1),\phi^{\prime}(\norm{\bfy}_1)], &  j\notin \supp(\bfy) \mbox{ and } p=1.\\
       \end{cases}
\]
\end{enumerate}
\end{lemma}

\begin{proof}
For brevity, denote the set $\prod_{j=1}^m S_j$ by $S$. In \textit{(i)}, let $\bfu\in \widehat{\partial}(\phi\circ g)(\bfy) $, which is the regular subdifferential at $\bfy=\bf0$. By the definition,
\[
\liminf\limits_{\substack{\bfz\to \bf0 \\ \bfz\neq\bf0}}\frac{\norm{\bfz}_p^q-<\bfu,\bfz-\bf0>}{\norm{\bfz-\bf0}_2}\geq 0.
\]
From the equivalence of norms when $p\geq 1$, we have
$$\norm{\bfz}_p\geq C\norm{\bfz}_2,$$
where $C>0$ is a constant. It is sufficient to have
\[
\frac{\norm{\bfz}_p^q-<\bfu,\bfz-\bf0>}{\norm{\bfz-\bf0}_2}\geq \frac{C^q\norm{\bfz}_2^q-<\bfu,\bfz>}{\norm{\bfz}_2}\geq 0,\hspace{2mm} \bfz\to\bf0.
\]
This is true for any $\bfu \in S$ due to $0<q<1$. Then the proof is finished by the fact that $\widehat{\partial}(\phi\circ g)(\bfy)\subseteq \partial (\phi\circ g)(\bfy)$.

In \textit{(ii)},~for $p>1$, the function $(\phi\circ g)(\bfy)$ is continuously differential at $\bfy$, so the subdifferential is the gradient in this case. For $p=1$, we show that $S=\widehat{\partial}(\phi\circ g)(\bfy)$ firstly.
~On one hand, let $\bfu\in \widehat{\partial}(\phi\circ g)(\bfy) $ and $\bfy\neq \bf0$, the limit inferior hold along the special direction,
\[
\liminf\limits_{\substack{\bfz_k=\bfy_{k},k\neq j \\ \bfz_j\to \bfy_{j}\\ \bfz\neq\bfy}}\frac{\norm{\bfz}_1^q-\norm{\bfy}_1^q-<\bfu,\bfz-\bfy>}{\norm{\bfz-\bfy}_2}\geq 0.
\]
Then we have
\[ \begin{cases}
(\bfu)_j=\phi^{\prime}(\norm{\bfy}_1)\cdot \sgn(\bfy_{j}),&j\in \supp(\bfy),\\
\abs{(\bfu)_j}\leq \phi^{\prime}(\norm{\bfy}_1), &j\notin \supp(\bfy),\\
\end{cases}
\]
by the differential mean value theorem. So $\widehat{\partial}(\phi\circ g)(\bfy)\subseteq S$.

On the other hand, we construct function $h(\bfz)$ when $\bfz$ is in the neighbourhood of $\bfy$:
\[
h(\bfz)=\left(\sum_{j\in \supp(\bfy)}\abs{\bfz_j}+\sum_{j\notin \supp(\bfy)}k_j\bfz_j\right)^q,
\]
where $k_j\in[ -1,1]$. Then $h$ is differentiable at $\bfy$ and $h(\bfz)\leq \phi(\norm{\bfz}_1)$ , $h(\bfy)=\phi(\norm{\bfy}_1)$. From ~\cite[Proposition 8.5]{Rockafellar2009Variational}, we have $\nabla h(\bfy)\in \widehat{\partial} (\phi\circ g)(\bfy)$. Here
\[
\left(\nabla h(\bfy)\right)_j=\begin{cases}
\phi^{\prime}(\norm{\bfy}_1)\cdot \sgn(\bfy_{j}),&j\in \supp(\bfy),\\
\phi^{\prime}(\norm{\bfy}_1)\cdot k_j,&j\notin  \supp(\bfy),\\
\end{cases}
\]
to obtain $S\subseteq\widehat{\partial}(\phi\circ g)(\bfy)$ by the arbitrary $k_j\in[-1,1],j\notin \supp(\bfy)$. Hence $S=\widehat{\partial}(\phi\circ g)(\bfy)$.

The left is to show $\partial (\phi\circ g)(\bfy)\subseteq\widehat{\partial}(\phi\circ g)(\bfy)$, since the inclusion relationship in the other direction holds from the \textit{remark} of Definition \ref{def-subdiff}.

In fact, suppose $\bfu \in \partial (\phi\circ g)(\bfy)$, by the definition, there exists $\bfz^{(k)} \to \bfy, \phi(\norm{\bfz^{(k)}}_1) \to \phi(\norm{\bfy}_1)$ and $\bfu^{(k)}\in \widehat{\partial} (\phi\circ g)(\bfz^{(k)}),\bfu^{(k)}\to \bfu$, thus $\bfz^{(k)}$ and $\bfy$ have the identical support when $k$ is sufficiently large. Based on it and from the fact
\[ \begin{cases}
(\bfu^{(k)})_j=\phi^{\prime}(\norm{\bfz^{(k)}}_1)\cdot \sgn(\bfz_j^{(k)}),&j\in \supp(\bfz^{(k)}),\\
\abs{(\bfu^{(k)})_j}\leq \phi^{\prime}(\norm{\bfz^{(k)}}_1), &j\notin \supp(\bfz^{(k)}),\\
\end{cases}
\]
we obtain that $\bfu\in \widehat{\partial}(\phi\circ g)(\bfy)$ by the limit process.
\end{proof}

The regularity property of function is essential for dealing with the subdifferential of the addition of two non-smooth norms, i.e. $\ell_{1,q}$ term and $\ell_1$/$\ell_{\infty}$ noise term, we give the lemma here.
\begin{lemma}[Regularity]\label{lemma:regular}
Let $\bfy\in\mathbb{R}^m$ be the $m$-dimensional vector, then $(\phi\circ g)(\bfy)$ is regular at $\bfy$ for $p\geq 1$.
\end{lemma}

\begin{proof}
By~\cite[Corollary 8.11]{Rockafellar2009Variational}, $(\phi\circ g)(\bfy)$ is regular at $\bfy$ if and only if
\begin{equation}\label{equiv-condition}
\partial(\phi\circ g)(\bfy)=\widehat{\partial}(\phi\circ g)(\bfy),\hspace{3mm} \partial^{\infty}(\phi\circ g)(\bfy)=(\widehat{\partial}(\phi\circ g)(\bfy))^{\infty}.
\end{equation}
In the proof of Lemma~\ref{lem-subdifferential}, we know that the first equality in~\eqref{equiv-condition} holds. The left is to verify the second equality.

For $\bfy=\bf0$, we have
\begin{equation*}
\widehat{\partial}(\phi\circ g)(\bfy)=(-\infty,\infty)^m,
\end{equation*}
thus the horizon cone $(\widehat{\partial}(\phi\circ g)(\bfy))^{\infty}$ is the same set $(-\infty,\infty)^m$ by letting $\bfv^{(k)}=k\bfv$ and $\lambda^{(k)}=1/k$ in Definition~\ref{def-horizon-cone}. We can also conclude the horizon subdifferential $\partial^{\infty}(\phi\circ g)(\bfy)=(-\infty,\infty)^m$ by the same trick.

For $\bfy\neq\bf0$, we have the following from Definition \ref{def-subdiff} and the {\em remark} of Definition \ref{def-horizon-cone}:
\begin{equation*}
\partial^{\infty}(\phi\circ g)(\bfy)=(\widehat{\partial}(\phi\circ g)(\bfy))^{\infty}=\{\bf0\},
\end{equation*}
due to the boundedness of $\widehat{\partial}(\phi\circ g)(\bfy)$.
\end{proof}

\begin{remark}
From~\cite[Proposition 10.5]{Rockafellar2009Variational} for separable functions, the sum function
\begin{equation*}
\norm{\bfx}_{p,q}^{q}=\sum_{\sfi \in \sfG}\phi(\norm{\bfx_{\sfi}}_p)
\end{equation*}
is also regular.
\end{remark}

The objective function $\calE$ in~\eqref{eq:lp-mini-prob} reads
\begin{equation}\label{eq:object-func}
	\calE(\bfx) = \sum_{\sfi \in \sfG}\phi(\norm{\bfx_{\sfi}}_p)  +F_r(\bfx),  ~p\geq 1,~1\leq r \leq \infty,
\end{equation}
which is bounded below, coercive, and continuous. It has at least one minimizer.

Now, we derive the subdifferential of $\calE$ at $\bfx$. From Lemma~\ref{lemma:regular} and the remark, we know that $\sum_{\sfi \in \sfG}\phi(\norm{\bfx_{\sfi}}_p)$ is regular. For $1\leq r\leq \infty$, $F_r(\bfx)$ is convex and also regular. By~\cite[Exercise 10.9]{Rockafellar2009Variational}, we get
\begin{equation}\label{eq:obj-func-subdiff-p2}
  \partial \calE(\bfx) = \partial \left(\sum_{\sfi \in \sfG} \phi(\norm{\bfx_{\sfi}}_p)\right)
  +\partial F_r(\bfx).
\end{equation}
The subdifferential on the first term in \eqref{eq:obj-func-subdiff-p2} can be obtained by~\cite[Proposition 10.5]{Rockafellar2009Variational},
\begin{equation}\label{eq:subdiff-first-term}
\partial \left(\sum_{\sfi \in \sfG} \phi(\norm{\bfx_{\sfi}}_p)\right)=\prod\limits_{\sfi \in \sfG} \partial (\phi\circ g) (\bfx_{\sfi}).
\end{equation}
The subdifferential factors in the right-hand term can be calculated by Lemma~\ref{lem-subdifferential} according to the specific cases of $\bfx_{\sfi}$. The subdifferential on the second term in \eqref{eq:obj-func-subdiff-p2} can be obtained by the chain rule of composite subdifferential,
\[ \displaystyle\partial F_r(\bfx)=\begin{cases}
 \displaystyle\frac{1}{\alpha r}\bfA^{T} \partial \|\bfv\|_r^r|_{\bfv=\bfA\bfx-\bfy},& r\geq 1,\\
\displaystyle\frac{1}{\alpha} \bfA^T \partial \|\bfv\|_{\infty}|_{\bfv=\bfA\bfx-\bfy},& r=\infty.
\end{cases}
\]
where the subdifferential of the infinity norm can be derived as follows.
From the Danskin-Bertsekas Theorem for subdifferential in ~\cite[Proposition A.22]{Bertsekas1971thesis}, it holds that
\begin{equation}\label{eq:infty-p2}
\partial \norm{\bfh}_{\infty}=\{\bfu\in\mathbb{R}^M ~|\norm{\bfu}_1\leq 1,\bfh^T\bfu=\norm{\bfh}_{\infty}\}.
\end{equation}
Hence, the each entry of element in $\partial F_r(\bfx)$, denoted by $\vec{\eta}_{\sfi,j}(\bfx),\sfi \in \sfG,j\in J_{\sfi}$ has the following representation,
\begin{equation}\label{eq:partialFr}
\vec{\eta}_{\sfi,j}(\bfx)=\begin{cases}
\dfrac{1}{\alpha}  \displaystyle\sum_{k \in I}\left(\abs{\sum_{\sfm\in\sfG}A_{k,\sfm} \bfx_{\sfm}-y_k}^{r-1}\cdot\sgn\left(\sum_{\sfm\in\sfG}A_{k,\sfm} \bfx_{\sfm}-y_k\right)\cdot (A_{k,\sfi})_j\right),& r>1,\\
\in \dfrac{1}{\alpha}\displaystyle\sum_{k\in I}\partial\abs{\cdot}(\sum_{\sfm\in\sfG}A_{k,\sfm}\bfx_{\sfm}-y_k)(A_{k,\sfi})_j,&r=1,\\
\in \displaystyle\frac{1}{\alpha}\left\{ \sum_{k\in I}(A_{k,\sfi})_j u_k~| \norm{\bfu}\leq 1, \bfu^T(\bfA\bfx-\bfy)=\norm{\bfA\bfx-\bfy}_{\infty}    \right\},&r=\infty.
\end{cases}
\end{equation}

From the definition of the subdifferential,  we have that $\bfx^{\ast}$ is a stationary point of ~\eqref{eq:lp-mini-prob} if and only if
\begin{equation}\label{eq:first-opt-cond}
       \bf0 \in \partial \calE(\bfx^{\ast}).
\end{equation}
\subsection{A motivating proposition}

The following proposition inspires us to design the algorithm in the next section.

\begin{proposition}\label{lem-motivation}
Suppose $\bfx \in \bbR^N$ has the group structure $\bfx:=\left(\bfx_{\sf1}^T,\bfx_{\sf2}^T,\cdots,\bfx_{\sfg}^T\right)^T$. If $\bfx$ is sufficiently close to a local minimizer (or a stationary point) $\bfx^{\ast}$ of~\eqref{eq:lp-mini-prob}.
Then it holds that
   \begin{equation}\label{eq:motivation}
     \bfx_{\sfi}^{\ast} = \bf0, \; \forall \sfi \in \sfG \setminus \supp_{\sfG}(\bfx).
   \end{equation}
\end{proposition}

\begin{proof}
We prove~\eqref{eq:motivation} by contradiction.

As $\bfx^{\ast}$ is a local minimizer (or a stationary point) of $\calE$, the condition~\eqref{eq:first-opt-cond} implies that $\bf0\in \partial \cal{E}(\bfx^{\ast})$.

If $\bfx_{\sfi^{\prime}}^{\ast}\neq \bf0,\sfi^{\prime}\in\sfG\setminus\supp_{\sfG}(\bfx)$, that is, $\bfx^{\ast}_{\sfi^{\prime},j}\neq 0$ for some $ j\in J_{\sfi^{\prime}} $. For $1 < r < \infty$, we have
    \begin{equation}\label{eq:first-opt-cond-component-q2}
        0 = \phi^{\prime}(\norm{\bfx_{\sfi^{\prime}}^{\ast}}_p)\norm{\bfx_{\sfi^{\prime}}^{\ast}}_p^{1-p}|\bfx^{\ast}_{\sfi^{\prime},j}|^{p-1}\sgn(\bfx^{\ast}_{\sfi^{\prime},j})
       + \vec{\eta}_{\sfi^{\prime},j}(\bfx^{\ast}).
    \end{equation}
Summing up all the absolute values of the two terms in~\eqref{eq:first-opt-cond-component-q2} for $j\in \supp(\bfx^{\ast}_{\sfi^{\prime}})$, we have
\begin{equation}\label{eq:first-opt-cond-component-q3}
q\norm{\bfx^{\ast}_{\sfi^{\prime}}}_p^{q-1}\leq q\norm{\bfx_{\sfi^{\prime}}^{\ast}}_p^{q-p}\norm{\bfx^{\ast}_{\sfi^{\prime}}}^{p-1}_{p-1}=\sum_{j\in \supp(\bfx_{\sfi^{\prime}}^{\ast})}\abs{\vec{\eta}_{\sfi^{\prime},j}(\bfx^{\ast})},
\end{equation}
the left inequality holds from Lemma \ref{lemma:norm} for $p>1$ and from
$\norm{\bfx^{\ast}_{\sfi^{\prime}}}^{p-1}_{p-1}=\# \{ \mbox{nonzero etries of } \bfx^{\ast}_{\sfi^{\prime}}\}$
for $p=1$.

The right side of~\eqref{eq:first-opt-cond-component-q3} is uniformly bounded in the neighborhood of $\bfx$, and the bound is independent of $\bfx^*$. Since $\bfx_{\sfi^{\prime}}^{\ast}$ can be sufficiently close to $\bfx_{\sfi^{\prime}} = 0, \sfi^{\prime}\in \sfG\setminus\supp_{\sfG}(\bfx)$, it contradicts ~\eqref{eq:first-opt-cond-component-q3} by $0 < q < 1$.

For $r=1$ and $r=\infty$, we have
\begin{equation}\label{eq:first-opt-cond-component-q4}
        0 \in \phi^{\prime}(\norm{\bfx_{\sfi^{\prime}}^{\ast}}_p)\norm{\bfx_{\sfi^{\prime}}^{\ast}}_p^{1-p}|\bfx^{\ast}_{\sfi^{\prime},j}|^{p-1}\sgn(\bfx^{\ast}_{\sfi^{\prime},j})
       + \vec{\eta}_{\sfi^{\prime},j}(\bfx^{\ast}),
    \end{equation}
Thus, the results can be derived similarly from the uniform boundedness of the sets $\vec{\eta}_{\sfi^{\prime},j}(\bfx^{\ast})$ in the neighborhood of $\bfx$.

\end{proof}

\begin{remark}
For the special case $r=2$ in fidelity term, \cite{Chen2010Lower,HuYaohuaetal2017GSO} established the lower bound theory, which can also inspire our proposition.
\end{remark}

\subsection{Algorithm}
Motivated by Proposition~\ref{lem-motivation}, we propose to solve the problem~\eqref{eq:lp-mini-prob} by an iterative process, which generates a sequence whose group support set is nonincreasing. Suppose that $\bfx^{(l)}$ is an approximate solution in the $l$th iteration. In the next iteration, we minimize the objective function only on the group support set $\sfS^{(l)}$ of $\bfx$, with the remaining group components being null. This idea yields the following iterative support shrinking algorithm (ISSA).

\vspace{3mm}
\begin{mdframed}[frametitle = {ISSA: Iterative Support Shrinking Algorithm}, frametitlerule = true]

\noindent{\bf Initialization:} Select $\bfx^{(0)} \in \bbR^{N}$.

\noindent{\bf Iteration:} For $l = 0, 1, \ldots$ until convergence:
\begin{enumerate}
\renewcommand{\labelenumi}{ \arabic{enumi}.}
  \item Set $\sfS^{(l)}  = \supp_{\sfG}(\bfx^{(l)})$.
  \item Compute $\bfx^{(l+1)}_{\sfS^{(l)}}$ by solving
\begin{equation}\label{eq:constrained-lp-lq-prob}
    \begin{aligned}
      &  \min_{\bfx_{\sfS^{(l)}}} \sum_{ \sfi \in \sfS^{(l)} } \phi( \norm{\bfx_{\sfi}}_p)
	+F_{r}^{(l)}(\bfx),
    \end{aligned}
    \tag{$\mathcal{P}_{o}$}
\end{equation}
where $F_r^{(l)}(\bfx)$ is the distance of $F_r(\bfx)$ at the $l$-th step over the group support set $\sfS^{(l)}$,
\begin{equation*}
F_{r}^{(l)}(\bfx)=\begin{cases}
\displaystyle\frac{1}{r\alpha}  \sum_{k \in I}\abs{\sum_{\sfi\in \sfS^{(l)}} A_{k,\sfi} \bfx_{\sfi} - y_k }^{r}, &r\geq 1,\\
\displaystyle\frac{1}{\alpha} \max_{k\in I}\abs{\sum_{\sfi\in \sfS^{(l)}} A_{k,\sfi} \bfx_{\sfi} - y_k },&r=\infty.
\end{cases}
\end{equation*}
\item Set $$\bfx_{\sfi}^{(l+1)}=\vec{0}, \mbox{ for }\sfi\in \sfG \setminus {\sfS}^{(l)}.$$
\end{enumerate}

\end{mdframed}

To make {\bf ISSA} more practical, each term $\phi( \norm{\bfx_{\sfi}}_p), \sfi \in \sfS^{(l)}$ can be linearized at $\|\bfx_{\sfi}^{(l)}\|_p \neq 0$. We introduce the following energy functional with proximal linearization:
\begin{equation}\label{eq:regularized-linear-approx-obj}
\begin{split}
   \calE^{(l)}(\bfx)  =&  \sum_{ \sfi \in \sfS^{(l)} } \phi(\|\bfx_{\sfi}^{(l)}\|_p)
	+ \phi^{\prime}(\|\bfx_{\sfi}^{(l)}\|_p) \left(\norm{\bfx_{\sfi}}_p - \|\bfx_{\sfi}^{(l)}\|_p\right)\\
	&+F_r^{(l)}(\bfx) + \frac{\beta}{2}\| \bfx - \bfx^{(l)} \|_{2}^{2}.
\end{split}	
\end{equation}
where $\beta\geq 0$.

We present an inexact iterative support shrinking algorithm with proximal linearization to solve~\eqref{eq:lp-mini-prob}.

\vspace{4mm}
\begin{mdframed} [frametitle = {InISSAPL: Inexact Iterative Support Shrinking Algorithm with Proximal Linearization}, frametitlerule = true]

\noindent {\bf Initialization:}  Select $\bfx^{(0)}=c\mathbb{1}$ with $c\neq 0$ or randomly, where $\mathbb{1}$ is the all one vector.

\noindent {\bf Iteration:} For $l = 0, 1, \ldots$ until convergence:
\begin{enumerate}
\renewcommand{\labelenumi}{ \arabic{enumi}.}
\renewcommand{\labelenumii}{ \arabic{enumi}.\arabic{enumii}.}
  \item Set $\sfS^{(l)}  = \supp_{\sfG}(\bfx^{(l)})$. Set $\beta=0$ for $l=0$ and $\beta>0$ fixed for $l\geq 1$.
  \item Compute $\bfx^{(l+1)}_{\sfS^{(l)}}$ by approximately solving
  \begin{equation}\label{eq:prob-z}
	 \min_{\bfx_{\sfS^{(l)}} } \calE^{(l)}(\bfx)\\
\tag{$\mathcal{P}_{x}$}
\end{equation}
such that
\begin{equation}\label{eq:sub-opt-cond}
\bfu^{(l)}(\bfx^{(l+1)})\in \partial \calE^{(l)}(\bfx^{(l+1)}), \|\bfu^{(l)}(\bfx^{(l+1)})\|_2\leq \frac{\beta}{2}\varepsilon \|\bfx^{(l+1)}-\bfx^{(l)}\|_2.
\end{equation}
with the tolerance error $\varepsilon$.
\item Set
$$\bfx_{\sfi}^{(l+1)}=\vec{0}, \bfu_{\sfi}^{(l)}(\bfx^{(l+1)})=0, \mbox{ for }\sfi\in \sfG \setminus {\sfS}^{(l)}.$$
\end{enumerate}
\end{mdframed}

\begin{remark}
The condition~\eqref{eq:sub-opt-cond} in InISSAPL is motivated by~\cite{Attouch2013Convergence,Zhifangliuaxiv}. It corresponds to an inexact inner loop and a guide to select the approximate solution for \eqref{eq:prob-z}. Due to the strong convexity of the problem \eqref{eq:prob-z}, it can be solved to any given accuracy. Therefore, the condition~\eqref{eq:sub-opt-cond} in InISSAPL can hold, as long as the problem \eqref{eq:prob-z} is solved sufficiently accurately.
\end{remark}

\begin{remark}
From the motivating Proposition \ref{lem-motivation}, $\bfx^{(0)}$ is required to be with as large support as possible. There are two strategies to choose the starting point. One is to set $\bfx^{(0)}$ by nonzero scalar multiplication of the all one vector, which yields a group lasso when $p=2$ for the first step. The other is to set $\bfx^{(0)}$ by randomly generating data of i.i.d Gaussian (with zero probability to obtain zero group member), indicating a weighted group lasso when $p=2$. Due to the fact that $\bfx^{(0)}$  is not the proximal solution, we also set $\beta=0$ for the first step in the algorithm. The results of experiments with suggested two kinds of starting points are given in section \ref{discuss-init}.

\end{remark}

For the convenience of description later, we give the representation of the subdifferential in~\eqref{eq:sub-opt-cond} for $~\sfi\in \sfS^{(l)},~j\in J_{\sfi}$,
\begin{equation}\label{eq:partial_calEL}
\begin{split}
\bfu_{\sfi,j}^{(l)}(\bfx)&=\vec{\zeta}_{\sfi,j}^{(l)}(\bfx)+\vec{\eta}^{(l)}_{\sfi,j}(\bfx)+\beta(\bfx_{\sfi,j}-\bfx^{(l)}_{\sfi,j}),
\end{split}
\end{equation}
where
\[
\vec{\zeta}_{\sfi,j}^{(l)}(\bfx)=
\begin{cases}
\phi^{\prime}(\|\bfx_{\sfi}^{(l)}\|_p)\norm{\bfx_{\sfi}}_p^{1-p}\abs{\bfx_{\sfi,j}}^{p-1}\sgn(\bfx_{\sfi,j}),&~p>1,\\
\phi^{\prime}(\|\bfx_{\sfi}^{(l)}\|_1) \sgn(\bfx_{\sfi,j}),&~p=1,\mbox{ and }j\in \supp(\bfx_{\sfi}),\\
\in [-\phi^{\prime}(\|\bfx_{\sfi}^{(l)}\|_1),\phi^{\prime}(\|\bfx_{\sfi}^{(l)}\|_1)], &~p=1,\mbox{ and }j\notin \supp(\bfx_{\sfi}),
\end{cases}
\]
and
\[
\vec{\eta}_{\sfi,j}^{(l)}(\bfx)=\begin{cases}
\dfrac{1}{\alpha}  \displaystyle\sum_{k \in I}\left(\abs{\sum_{\sfm\in\sfS^{(l)}}A_{k,\sfm} \bfx_{\sfm}-y_k}^{r-1}\cdot\sgn\left(\sum_{\sfm\in\sfS^{(l)}}A_{k,\sfm} \bfx_{\sfm}-y_k\right)\cdot (A_{k,\sfi})_j\right),& r>1,\\
\in \dfrac{1}{\alpha}\displaystyle\sum_{k\in I}\partial\abs{\cdot}(\sum_{\sfm\in\sfS^{(l)}}A_{k,\sfm}\bfx_{\sfm}-y_k)(A_{k,\sfi})_j,&r=1,\\
\in \displaystyle\frac{1}{\alpha}\left\{ \sum_{k\in I}(A_{k,\sfi})_j u_k~| \norm{\bfu}\leq 1, \bfu^T(\bfA\bfx-\bfy)=\norm{\bfA\bfx-\bfy}_{\infty}    \right\},&r=\infty.
\end{cases}
\]

\section{Convergence analysis}
\label{sec:con-ana}

In this section, we establish the global convergence result of the sequence generated by the InISSAPL algorithm. Theorem~\ref{thm-convergence} in the appendix gives a celebrating theoretical framework for the convergence of sequence in decent methods. Recently it has extensive applications~\cite{Attouch2010Proximal, Attouch2013Convergence, Bolte2014Proximal}, especially in non-convex optimization. When we turn back to our problem, the key issue is to deal with the non-Lipschitz property of $\mathcal{E}(\bfx)$. In this paper, a lower bound theory of the iterative sequence is developed to overcome the difficulty of the non-Lipschitz property. Furthermore, due to the non-smooth property of $\mathcal{E}(\bfx)$, the construction of the element in $\partial \mathcal{E}(\bfx)$ to prove the relative error condition {\em (H2)} in Theorem~\ref{thm-convergence} is more technical.

From the iteration process, we can see that it produces a nonincreasing sequence of group support set. The lemma is given in the following.

\begin{lemma}\label{lem-finite-num-iter}
The sequence $\set{\sfS^{(l)}}$ converges in a finite number of iterations, i.e., there exists an integer $ L > 0$ such that if $l \geq L$, then $\sfS^{(l)} \equiv  \sfS^{(L)}$.
\end{lemma}

\begin{proof}
Since $\sfG$ is a finite set and
\[
  \sfG \supseteq \sfS^{(0)} \supseteq \cdots \supseteq \sfS^{(l)} \supseteq \cdots,
\]
 $\set{\sfS^{(l)}}$ converges in a finite number of iterations.
\end{proof}

In the next, we verify the conditions {\em (H1)-(H3)} in Theorem~\ref{thm-convergence} for the sequence of the objective function $\mathcal{E}(\bfx^{(l)})$. {\em (H1)} is the sufficient decrease condition for the sequence, and it is given in Lemma~\ref{lem-bound-sufficient-decrease}. Here we introduce the energy functional with proximal linearization once again, but defined over $\bfx\in \mathbb{R}^N$:
\begin{equation}\label{eq:prox-linear-approx-obj}
\begin{split}
   \calF^{(l)}(\bfx) & =  \sum_{ \sfi \in \sfS^{(l)} } \phi(\|\bfx_{\sfi}^{(l)}\|_p)
	+ \phi^{\prime}(\|\bfx_{\sfi}^{(l)}\|_p) \left(\|\bfx_{\sfi}\|_p - \|\bfx_{\sfi}^{(l)}\|_p\right)
	+ F_r(\bfx)+ \frac{\beta}{2}\| \bfx - \bfx^{(l)} \|_{2}^{2}.
\end{split}	
\end{equation}
It should be noted that it is different from $\mathcal{E}^{(l)}(\bfx)$ in (\ref{eq:regularized-linear-approx-obj}) by the fidelity term.

\begin{lemma}\label{lem-bound-sufficient-decrease}
For any $\beta > 0$ and $0 \leq \ve < 1$, let $\set{\bfx^{(l)}}$ be a sequence generated by {\rm InISSAPL}. Then
\begin{enumerate}
\item The sequence $\set{\calE(\bfx^{(l)})}$ is nonincreasing and satisfies
    \begin{equation}\label{eq:sufficient-decrease-condition}
    \calE(\bfx^{(l+1)})+ \frac{\beta}{2}(1-\ve)\| \bfx^{(l+1)} - \bfx^{(l)} \|_{2}^{2} \leq \calE(\bfx^{(l)}).
    \end{equation}
\item The sequence $\set{\bfx^{(l)}}$  is bounded and satisfies $\lim_{l \to \infty} \| \bfx^{(l+1)} - \bfx^{(l)} \|_{2} = 0$.
\end{enumerate}
\end{lemma}

\begin{proof}
Due to the fact that $\phi(0) = 0$, we have
\begin{equation}\label{eq:mm-eq1}
\begin{split}
	\calF^{(l)}(\bfx^{(l)}) & =\sum_{ \sfi \in \sfS^{(l)} } \phi(\|\bfx_{\sfi}^{(l)}\|_p) + F_r(\bfx^{(l)}) \\
	& =\sum_{ \sfi \in \sfG } \phi(\|\bfx_{\sfi}^{(l)}\|_p) + F_r(\bfx^{(l)}) = \calE(\bfx^{(l)}).
\end{split}
\end{equation}
When $\bfx \in \bbR^{N}$ and $\supp_G(\bfx) \subseteq \sfS^{(l)}$, we obtain
\begin{equation}\label{eq:mm-eq2}
\begin{split}
	\calF^{(l)}(\bfx) &= \sum_{ \sfi \in \sfS^{(l)} } \phi(\|\bfx_{\sfi}^{(l)}\|_p)
	+ \phi^{\prime}(\|\bfx_{\sfi}^{(l)}\|_p) \left(\|\bfx_{\sfi}\|_p - \|\bfx_{\sfi}^{(l)}\|_p\right)
+ F_r(\bfx)
	+ \frac{\beta}{2}\| \bfx - \bfx^{(l)} \|_{2}^{2}
\\
[~\text{by~\eqref{eq:first-approx}}~]~  &\geq  \sum_{ \sfi \in \sfS^{(l)} } \phi(\|\bfx_{\sfi}\|_p) 	+ F_r(\bfx) + \frac{\beta}{2}\| \bfx - \bfx^{(l)} \|_{2}^{2} \\
	&=  \sum_{\sfi \in \sfG } \phi(\|\bfx_{\sfi}\|_p) 	+ F_r(\bfx)
	+ \frac{\beta}{2}\| \bfx - \bfx^{(l)} \|_{2}^{2} \\
	&= \calE(\bfx)+ \frac{\beta}{2}\| \bfx - \bfx^{(l)} \|_{2}^{2}.
\end{split}
\end{equation}

Let $\widehat{\bfu}^{(l)}(\bfx)\in \partial \calF^{(l)}(\bfx)$. Then
\begin{equation}
\widehat{\bfu}_{\sfi,j}^{(l)}(\bfx)=\begin{cases}
\bfu_{\sfi,j}^{(l)}(\bfx),& \sfi\in \sfS^{(l)},j\in J_{\sfi},\\
\vec{\eta}_{\sfi,j}(\bfx)+\beta(\bfx-\bfx^{(l)}),&\sfi\in \sfG\setminus\sfS^{(l)},j\in J_{\sfi},
\end{cases}
\end{equation}
where $\bfu_{\sfi,j}^{(l)}(\bfx)$ is defined in~\eqref{eq:partial_calEL} and $\vec{\eta}_{\sfi,j}(\bfx)$ is defined in~\eqref{eq:partialFr}.
Since for any $ \sfi \in \sfG\setminus  \sfS^{(l)}$, $\bfx_{\sfi}^{(l+1)} = \bfx_{\sfi}^{(l)} = \bf{0}$, we have
\begin{equation}\label{eq:mm-eq3}
\begin{split}
	\langle \widehat{\bfu}^{(l)}(\bfx^{(l+1)}) , \bfx^{(l)} - \bfx^{(l+1)} \rangle
	& = \sum_{\sfi\in \sfS^{(l)}}\sum_{j \in J_{\sfi} } \widehat{\bfu}_{\sfi,j}^{(l)}(\bfx^{(l+1)})(\bfx_{\sfi,j}^{(l)} - \bfx_{\sfi,j}^{(l+1)})\\
	&\geq - \|\bfu^{(l)}_{\sfS^{(l)}}(\bfx^{(l+1)}) \|_{2}\cdot \|\bfx^{(l)}_{\sfS^{(l)}} - \bfx^{(l+1)}_{\sfS^{(l)}} \|_{2}\\
[~\text{by~\eqref{eq:sub-opt-cond}}~]~	&\geq - \frac{\beta}{2}\ve \| \bfx^{(l)} - \bfx^{(l+1)}\|_{2}^{2}.
\end{split}
\end{equation}

Putting~\eqref{eq:mm-eq1}, \eqref{eq:mm-eq2} and~\eqref{eq:mm-eq3} together, we obtain
\[
\begin{split}
	\calE(\bfx^{(l)}) = \calF^{(l)}(\bfx^{(l)}) & \geq \calF^{(l)}(\bfx^{(l+1)}) + \langle \widehat{\bfu}^{(l)}(\bfx^{(l+1)}) , \bfx^{(l)} - \bfx^{(l+1)} \rangle \\
	& \geq \calF^{(l)}(\bfx^{(l+1)}) - \frac{\beta}{2}\ve \| \bfx^{(l+1)} - \bfx^{(l)} \|_{2}^{2}\\
	& \geq \calE(\bfx^{(l+1)})+ \frac{\beta}{2}(1-\ve)\| \bfx^{(l+1)} - \bfx^{(l)} \|_{2}^{2}.
\end{split}	
\]
With the fact that $\calE(\bfx)$ is bounded from below and $ \frac{\beta}{2}(1-\ve) > 0$, it follows that $\set{\calE(\bfx^{(l)})}$ is nonincreasing and converges to a finite value as $l \to \infty$. Thus
\[
\lim_{l\to \infty} \| \bfx^{(l+1)} - \bfx^{(l)} \|_{2} = 0.
\]
Because $\calE(\bfx)$ is coercive, we know that $\set{\bfx^{(l)}}$ is bounded.

\end{proof}

The following lemma is the lower bound theory on the nonzero groups of the iteration sequence, which can be used to overcome the non-Lipschitz property.
\begin{lemma}\label{lem-bound-theory}
There are $ 0 < c < C < \infty, L>0$ such that
\begin{equation}\label{eq:bound-theory}
	either \quad \bfx_{\sfi}^{(l)} = 0  \quad or \quad c \leq \|\bfx_{\sfi}^{(l)}\|_p \leq C, \; \forall \sfi \in \sfG, \; \forall l \geq L.
\end{equation}
\end{lemma}

\begin{proof}
From Lemma~\ref{lem-finite-num-iter}, for any $\sfi \in \sfS^{(L)}$ and $l \geq L$, $\bfx_{\sfi}^{(l)} \neq \bf0$. The sequence has upper bound from Lemma \ref{lem-bound-sufficient-decrease},
\begin{equation*}
 \|\bfx_{\sfi}^{(l)}\|_p \leq C.
\end{equation*}
We now prove by contradiction that $\|\bfx_{\sfi}^{(l)}\|_p $ has nonzero lower bound for any $\sfi \in \sfS^{(L)}, l \geq L$.

Suppose there exists $\sfi^{\prime} \in \sfS^{(L)}$ for some subsequence $\bfx^{(l_k)}$, still denoted by $\bfx^{(l)}$, such that
\[
  \bfx_{\sfi^{\prime}}^{(l)} \neq {\bf0} \text{ and } \lim_{ l \to \infty} \bfx_{\sfi^{\prime}}^{(l)} = {\bf0}.
\]
By the subdifferential expression ~\eqref{eq:partial_calEL}, we have for $j\in \supp(\bfx_{\sfi^{\prime}}^{(l+1)})$, and $p\geq 1$,
\begin{equation}\label{eq:bound-entry-triangle-K}
\abs{\zeta_{\sfi^{\prime},j}^{(l)}(\bfx^{(l+1)})}\leq \abs{ \bfu_{\sfi^{\prime},j}^{(l)}(\bfx^{(l+1)})}+\abs{\eta_{\sfi^{\prime},j}^{(l)}(\bfx^{(l+1)})}+\beta\abs{\bfx_{\sfi^{\prime},j}^{(l+1)}-\bfx_{\sfi^{\prime},j}^{(l)}}
\end{equation}
with the left term,
\[
\phi^{\prime}(\|\bfx_{\sfi^{\prime}}^{(l)}\|_p)\cdot\|\bfx_{\sfi^{\prime}}^{(l+1)}\|_p^{1-p}\cdot|\bfx_{\sfi^{\prime},j}^{(l+1)}|^{p-1}\leq |\zeta_{\sfi^{\prime},j}^{(l)}(\bfx^{(l+1)})|. \]
Summing up all the terms for $j\in \supp(\bfx_{\sfi^{\prime}}^{(l+1)})$, we have
\begin{eqnarray*}
\sum_{j\in \supp(\bfx_{\sfi^{\prime}}^{(l+1)})}\abs{ \bfu_{\sfi^{\prime},j}^{(l)}(\bfx^{(l+1)})}+\abs{\eta_{\sfi^{\prime},j}^{(l)}(\bfx^{(l+1)})}+\beta\abs{\bfx_{\sfi^{\prime},j}^{(l+1)}-\bfx_{\sfi^{\prime},j}^{(l)}}&\geq&\phi^{\prime}(\norm{\bfx_{\sfi^{\prime}}^{(l)}}_p)\cdot\norm{\bfx_{\sfi^{\prime}}^{(l+1)}}_p^{1-p}\cdot\norm{\bfx_{\sfi^{\prime}}^{(l+1)}}^{p-1}_{p-1} \\
 &\geq& \phi^{\prime}(\|\bfx_{\sfi^{\prime}}^{(l)}\|_p)\\
 &=& q \|\bfx_{\sfi^{\prime}}^{(l)}\|_p^{q-1},
\end{eqnarray*}
where the second inequality holds from the same reason as the motivating proposition (Proposition~\ref{lem-motivation}).  It follows from the boundedness of $\set{\bfx^{(l)}}$ that $\abs{\eta_{\sfi^{\prime},j}^{(l)}(\bfx^{(l+1)})}+\beta\abs{\bfx_{\sfi^{\prime},j}^{(l+1)}-\bfx_{\sfi^{\prime},j}^{(l)}}$ is bounded.
The condition~\eqref{eq:sub-opt-cond} implies that $\abs{\bfu_{\sfi^{\prime},j}^{(l)}(\bfx^{(l+1)} ) }$ is also bounded. Thus the equation~\eqref{eq:bound-entry-triangle-K} is impossible to hold when $l \to \infty$ because of $0<q<1$.

\end{proof}

By combining Lemma~\ref{lem-bound-theory} and~Proposition \ref{prop-phi}, we can obtain the Lipschitz property over the support of group members.
\begin{equation}\label{eq:grad-Lip-cond-K}
	\abs{ \phi^{\prime}(\|\bfx_{\sfi}^{(l+1)}\|_p) - \phi^{\prime}(\|\bfx_{\sfi}^{(l)}\|_p) }
	\leq L_{c}\abs{\|\bfx_{\sfi}^{(l+1)}\|_p - \|\bfx_{\sfi}^{(l)}\|_p}
	\leq L_{c}\|\bfx_{\sfi}^{(l+1)}-\bfx_{\sfi}^{(l)}\|_p,~\sfi\in \sfS^{(L)},~l\geq L
\end{equation}
when $ p\geq 1$.

Using this property, we can prove the relative error condition {\em (H2)} by Lemma \ref{lem-relative-error-condition} in which the sequence $\bfv^{(l+1)}$ of $\partial \mathcal{E}(\bfx^{(l+1)})$ is well constructed though $\mathcal{E}(\bfx)$ is non-smooth.

\begin{lemma}\label{lem-relative-error-condition}
For each $l \geq L$, there exists $\bfv^{(l+1)}  \in \partial \calE(\bfx^{(l+1)})$  and constant $\widetilde{C}>0$ such that
\begin{equation}\label{eq:relative-error-condition}
	\| \bfv^{(l+1)} \|_{2}
	\leq \widetilde{C}\|\bfx^{(l+1)} - \bfx^{(l)} \|_{2}.
\end{equation}
\end{lemma}

\begin{proof}
For $l\geq L$, the vector $\bfu^{(l)}(\bfx^{(l+1)})$ in the set of $\partial \mathcal{E}^{(l)}(\bfx^{(l+1)}) $ has the form in ~\eqref{eq:partial_calEL},
\begin{equation*}
\bfu_{\sfi,j}^{(l)}(\bfx^{(l+1)})=\vec{\zeta}_{\sfi,j}^{(l)}(\bfx^{(l+1)})+\vec{\eta}^{(l)}_{\sfi,j}(\bfx^{(l+1)})+\beta(\bfx_{\sfi,j}^{(l+1)}-\bfx_{\sfi,j}^{(l)}), ~\sfi\in \sfS^{(L)},~j\in J_{\sfi}.
\end{equation*}
Then the intermediate variable $\widehat{\bfv}^{(l+1)}$ is introduced as follows,
\[
\widehat{\bfv}^{(l+1)}_{\sfi,j}=\begin{cases}
\vec{\zeta}_{\sfi,j}^{(l)}(\bfx^{(l+1)})+\vec{\eta}^{(l)}_{\sfi,j}(\bfx^{(l+1)}),&\sfi\in \sfS^{(L)},j\in J_{\sfi},\\
\vec{0},&\sfi\in \sfG\setminus\sfS^{(L)},j\in J_{\sfi}.
\end{cases}
\]
The upper bound of $\widehat{\bfv}^{(l+1)}$ can be measured by the iterative error,
\begin{equation}\label{eq:subdiff-bound-eq1}
\begin{split}
	\| \widehat{\bfv}^{(l+1)} \|_{2} &= \sqrt{\sum_{\sfi \in \sfS^{(L)}}\sum_{j\in J_{\sfi}} |\widehat{\bfv}_{\sfi,j}^{(l+1)}|^{2}}
	 = \sqrt{\sum_{\sfi \in \sfS^{(L)}}\sum_{j\in J_{\sfi}} |\bfu_{\sfi,j}^{(l)}(\bfx^{(l+1)}) - \beta (\bfx_{\sfi,j}^{(l+1)} - \bfx_{\sfi,j}^{(l)}) |^{2}} \\
	& \leq \| \bfu^{(l)}_{\sfS^{(L)}}(\bfx^{(l+1)}) \|_{2} + \beta \|\bfx^{(l+1)} - \bfx^{(l)} \|_{2} \\
[~\text{by~\eqref{eq:sub-opt-cond}}~]~	& \leq \frac{\beta}{2}( \ve + 2 ) \|\bfx^{(l+1)} - \bfx^{(l)} \|_{2}.
\end{split}
\end{equation}

Noting the difference of $\partial \mathcal{E}^{(l)}(\bfx^{(l+1)})$ and $\partial \mathcal{E}(\bfx^{(l+1)})$, we specially construct $\bfv^{(l+1)}$ to be the form,
\[
{\bfv}^{(l+1)}_{\sfi,j}=\begin{cases}
\vec{\zeta}_{\sfi,j}(\bfx^{(l+1)})+\vec{\eta}^{(l)}_{\sfi,j}(\bfx^{(l+1)}),&\sfi\in \sfS^{(L)},j\in J_{\sfi},\\
\vec{0},&\sfi\in \sfG\setminus\sfS^{(L)},j\in J_{\sfi}.
\end{cases}
\]
where $\vec{\eta}^{(l)}_{\sfi,j}(\bfx^{(l+1)})$ is the same as the part of $\widehat{\bfv}_{\sfi,j}^{(l+1)}$ and
\begin{equation}\label{zetaij}
\vec{\zeta}_{\sfi,j}(\bfx^{(l+1)})=\begin{cases}
\phi^{\prime}(\|\bfx_{\sfi}^{(l+1)}\|_p)\|\bfx_{\sfi}^{(l+1)}\|_p^{1-p}|\bfx_{\sfi,j}^{(l+1)}|^{p-1}\cdot\sgn(\bfx_{\sfi,j}^{(l+1)}),&\sfi\in \sfS^{(L)},p>1,\\
\phi^{\prime}(\|\bfx_{\sfi}^{(l+1)}\|_1)\sgn(\bfx_{\sfi,j}^{(l+1)}), &\sfi\in \sfS^{(L)},p=1, j\in \supp(\bfx_{\sfi}^{(l+1)}),\\
\psi_{\sfi,j},&\sfi\in \sfS^{(L)},p=1, j\notin \supp(\bfx_{\sfi}^{(l+1)}).
\end{cases}
\end{equation}
Here $\psi_{\sfi,j}$ in $\vec{\zeta}_{\sfi,j}(\bfx^{(l+1)})$ is to be defined by the requirement of $\bfv^{(l+1)}\in \partial\mathcal{E}(\bfx^{(l+1)})$. On one hand, by Lemma \ref{lem-subdifferential} {\em (i)} and  (\ref{eq:obj-func-subdiff-p2})-(\ref{eq:subdiff-first-term}), for $\sfi\in\sfG\setminus \sfS^{(L)}$, $\partial(\phi\circ g)(\bfx_{\sfi})=\Pi_{j\in J_{\sfi}}(-\infty,+\infty)$ and the set $\partial F_r(\bfx^{(l+1)})$ is bounded, then $\bfv_{\sfi}^{(l+1)}$ belongs to the corresponding entries of the element in  $\partial \mathcal{E}(\bfx^{(l+1)})$. On the other hand, by Lemma \ref{lem-subdifferential} {\em (ii)} and (\ref{eq:obj-func-subdiff-p2})-(\ref{eq:subdiff-first-term}), for $\sfi\in \sfS^{(L)}$, it can be checked that if $\psi_{\sfi,j}$ satisfies $|\psi_{\sfi,j}|\leq q\|\bfx_{\sfi}^{(l+1)}\|_1^{q-1}$, ~$\vec{\zeta}_{\sfi}(\bfx)$ will be in $\partial (\phi\circ g)(\bfx_{\sfi})$. Thus $\bfv_{\sfi}^{(l+1)}$ also belongs to the corresponding entries of the element in  $\partial \mathcal{E}(\bfx^{(l+1)})$. Therefore, the left is to construct $\psi_{\sfi,j}$. It is more technical. $\psi_{\sfi,j}$ is determined by estimating the $\ell^1$ error of $\bfv^{(l+1)}$ and $\widehat{\bfv}^{(l+1)}$ in the case of $p=1,j\notin \supp(\bfx_{\sfi}^{(l+1)})$ later. Thus, the main idea of constructing $\psi_{\sfi,j}$ is to compare $\psi_{\sfi,j}\in[-q\|\bfx_{\sfi}^{(l+1)}\|_1^{q-1},q\|\bfx_{\sfi}^{(l+1)}\|_1^{q-1}]$ and $\vec{\zeta}_{\sfi,j}^{(l)}(\bfx^{(l+1)})\in[-q\|\bfx_{\sfi}^{(l)}\|_1^{q-1},q\|\bfx_{\sfi}^{(l)}\|_1^{q-1}]$ in (\ref{eq:partial_calEL}). That is, let $I=[-q\|\bfx_{\sfi}^{(l+1)}\|_1^{q-1},q\|\bfx_{\sfi}^{(l+1)}\|_1^{q-1}]$, if $\vec{\zeta}_{\sfi,j}^{(l)}(\bfx^{(l+1)})\in I$, we choose it. Otherwise, we choose the nearest point in $I$. Hence we choose
\begin{equation}\label{psiij}
\psi_{\sfi,j}=\begin{cases}
\vec{\zeta}_{\sfi,j}^{(l)}(\bfx^{(l+1)}),&\mbox{ if }\|\bfx_{\sfi}^{(l)}\|_1\geq \|\bfx_{\sfi}^{(l+1)}\|_1,\\
\vec{\zeta}_{\sfi,j}^{(l)}(\bfx^{(l+1)}),&\mbox{ if }\|\bfx_{\sfi}^{(l)}\|_1<\|\bfx_{\sfi}^{(l+1)}\|_1 \mbox{ and }|\vec{\zeta}_{\sfi,j}^{(l)}(\bfx^{(l+1)})|\leq q\|\bfx_{\sfi}^{(l+1)}\|_1^{q-1},\\
-q\|\bfx_{\sfi}^{(l+1)}\|_1^{q-1} ,  &\mbox{ if }\|\bfx_{\sfi}^{(l)}\|_1<\|\bfx_{\sfi}^{(l+1)}\|_1 \mbox{ and }\vec{\zeta}_{\sfi,j}^{(l)}(\bfx^{(l+1)})\in\left[-q\|\bfx_{\sfi}^{(l)}\|_1^{q-1},-q\|\bfx_{\sfi}^{(l+1)}\|_1^{q-1}\right),\\
q\|\bfx_{\sfi}^{(l+1)}\|_1^{q-1} ,  &\mbox{ if }\|\bfx_{\sfi}^{(l)}\|_1<\|\bfx_{\sfi}^{(l+1)}\|_1 \mbox{ and }\vec{\zeta}_{\sfi,j}^{(l)}(\bfx^{(l+1)})\in\left(q\|\bfx_{\sfi}^{(l+1)}\|_1^{q-1},q\|\bfx_{\sfi}^{(l)}\|_1^{q-1}\right].\\
\end{cases}
\end{equation}
where $\vec{\zeta}_{\sfi,j}^{(l)}(\bfx^{(l+1)})$ is the part of $\bfu_{\sfi,j}^{(l)}(\bfx^{(l+1)})$.
Noting that $0<q<1$, we can check that $|\psi_{\sfi,j}|\leq q\|\bfx_{\sfi}^{(l+1)}\|_1^{q-1}$.

After constructing $\vec\zeta_{\sfi,j}(\bfx^{(l+1)})$, we can now measure the difference between $\bfv^{(l+1)}$ and $\widehat{\bfv}^{(l+1)}$. We divide this measurement into two cases: $p>1$ and $p=1$. For $p>1$, the $L^1$ norm of the difference can be bounded by
\begin{equation}\label{eq:subdiff-bound-eq2}
\begin{split}
\| \bfv^{(l+1)} - \widehat{\bfv}^{(l+1)} \|_{1} =& \sum_{\sfi \in \sfS^{(L)}}\sum_{j\in J_{\sfi}}\abs{\vec{\zeta}_{\sfi,j}(\bfx^{(l+1)})- \vec{\zeta}_{\sfi,j}^{(l)}(\bfx^{(l+1)})  }      \\
    =& \sum_{\sfi \in \sfS^{(L)}}\sum_{j\in J_{\sfi}}\abs{\phi^{\prime}(\norm{\bfx_{\sfi}^{(l+1)}}_p)-\phi^{\prime}(\norm{\bfx_{\sfi}^{(l)}}_p)   }\cdot\norm{\bfx_{\sfi}^{(l+1)}}_p^{1-p}\cdot \abs{\bfx_{\sfi,j}^{(l+1)}}^{p-1}\\
    =&\sum_{\sfi \in \sfS^{(L)}}\abs{\phi^{\prime}(\norm{\bfx_{\sfi}^{(l+1)}}_p)-\phi^{\prime}(\norm{\bfx_{\sfi}^{(l)}}_p)   }\cdot\norm{\bfx_{\sfi}^{(l+1)}}_p^{1-p}\cdot \norm{\bfx_{\sfi}^{(l+1)}}^{p-1}_{p-1}\\
[~\text{by~\eqref{eq:grad-Lip-cond-K} and Lemma ~\ref{lemma:norm-equivalence} }~]~	  \leq & L_{c}\cdot C_s \sum_{\sfi\in \sfS^{(L)}}  \norm{\bfx_{\sfi}^{(l+1)}-\bfx_{\sfi}^{(l)}}_p \cdot\norm{\bfx_{\sfi}^{(l+1)}}_p^{1-p}\cdot \norm{\bfx_{\sfi}^{(l+1)}}^{p-1}_{p} \\
    \leq& L_{c}\cdot C_s\cdot C_p \|\bfx^{(l+1)} - \bfx^{(l)} \|_{2},
    \end{split}
\end{equation}
where $C_p$ is also the coefficient of norm equivalence. For $p=1$, it follows,
\begin{equation}\label{eq:subdiff-bound-eq3}
\begin{split}
\| \bfv^{(l+1)} - \widehat{\bfv}^{(l+1)} \|_{1}=&\sum_{\sfi \in \sfS^{(L)}}\sum_{j\in \supp(\bfx_{\sfi}^{(l+1)})}\abs{\vec{\zeta}_{\sfi,j}(\bfx^{(l+1)})- \vec{\zeta}_{\sfi,j}^{(l)}(\bfx^{(l+1)})}\\
&+\sum_{\sfi \in \sfS^{(L)}}\sum_{j\notin \supp(\bfx_{\sfi}^{(l+1)})}\abs{\psi_{\sfi,j}- \vec{\zeta}_{\sfi,j}^{(l)}(\bfx^{(l+1)})}\\
\leq & \sum_{\sfi \in \sfS^{(L)}}\sum_{j\in \supp(\bfx_{\sfi}^{(l+1)})}\abs{\phi^{\prime}(\norm{\bfx_{\sfi}^{(l+1)}}_1)-\phi^{\prime}(\norm{\bfx_{\sfi}^{(l)}}_1)   }\cdot\abs{\sgn(\bfx_{\sfi,j}^{(l+1)})}\\
&+\sum_{\sfi \in \sfS^{(L)}}\sum_{j\notin \supp(\bfx_{\sfi}^{(l+1)})}\abs{q\|\bfx_{\sfi}^{(l+1)}\|_1^{q-1}-q\|\bfx_{\sfi}^{(l)}\|_1^{q-1}   }\\
=& \sum_{\sfi \in \sfS^{(L)}}\sum_{j\in J_{\sfi}}\abs{\phi^{\prime}(\norm{\bfx_{\sfi}^{(l+1)}}_1)-\phi^{\prime}(\norm{\bfx_{\sfi}^{(l)}}_1)   }\\
\leq & L_c \|\bfx^{(l+1)} - \bfx^{(l)} \|_{1}\\
\leq & L_c\cdot C_p \|\bfx^{(l+1)} - \bfx^{(l)} \|_{2}.
\end{split}
\end{equation}
where the first inequality comes from (\ref{zetaij}), (\ref{psiij}) and (\ref{eq:partial_calEL}).

Combining~\eqref{eq:subdiff-bound-eq1},~\eqref{eq:subdiff-bound-eq2} and \eqref{eq:subdiff-bound-eq3} yields:
\[
\begin{split}
	\|\bfv^{(l+1)}\|_2\leq \| \bfv^{(l+1)} \|_{1} & \leq \| \bfv^{(l+1)} - \widehat{\bfv}^{(l+1)} \|_{1} +\sqrt{N} \| \widehat{\bfv}^{(l+1)} \|_{2} \\
	& \leq \widetilde{C} \|\bfx^{(l+1)} - \bfx^{(l)} \|_{2},
\end{split}
\]
where $\widetilde{C}=\max\{L_c C_s C_p,L_c C_p,\sqrt{N}\beta(2+\varepsilon)/2\}$.

\end{proof}

{\em (H3)} is the continuity condition, and it holds naturally. From Appendix~\ref{sec:appendix}, we know that $\mathcal{E}(\bfx)$ satisfies KL property. Finally, we establish our main convergence result.  


\begin{theorem}\label{thm-global-convergence}
The iterative sequence $\set{\bfx^{(l)}}$  generated by {\rm InISSAPL} algorithm converges globally to the limit point $\bfx^{\ast}$, which is a stationary point of problem~\eqref{eq:lp-mini-prob}.
\end{theorem}

\begin{proof}

Since $\set{\bfx^{(l)}}$ is bounded (Lemma~\ref{lem-bound-theory}), there exists a subsequence $(\bfx^{(k_{l})})$ and $\bfx^{\ast}$ such that
\begin{equation}\label{eq:continuity-condition}
  \bfx^{(k_{l})} \to \bfx^{\ast} \text{ and } \calE(\bfx^{(k_{l})}) \to \calE(\bfx^{\ast}),
\text{as } l \to \infty.
\end{equation}

By combing~\eqref{eq:sufficient-decrease-condition},~\eqref{eq:relative-error-condition} and~\eqref{eq:continuity-condition}, and by Theorem \ref{thm-convergence} in the appendix, the sequence $\set{\bfx^{(l)}}$ converges globally to the limit point $\bfx^{\ast}$, which is a stationary point of $\calE$.
\end{proof}

\section{Algorithm Implementation}
\label{sec:imlpem}

For each iteration step in InISSAPL algorithm, it is a weighted $\ell_{p,1}-\ell_{r} (~p\geq 1,~r\geq 1)$ minimization in essence. It is convex and the  inexact inner loop is allowed in implementation. Some standard methods like ADMM~\cite{Boyd2011Distributed}, split Bregman method~\cite{Goldstein2009Split,Yin2008Bregman} and primal-dual algorithm~\cite{Chambolle2011first,Esser2010General} can be used to efficiently solve it. Here we adopt scaled ADMM.

\subsection{Scaled ADMM}
a
At each $l$-th step in InISSAPL, it is equivalently to solving~\eqref{eq:prob-z} by
\begin{equation}\label{eq:alg-supproblem}
  \min_{\bfx_{\sfS^{(l)}}} \sum_{ \sfi \in \sfS^{(l)} } \phi^{\prime}(\|\bfx_{\sfi}^{(l)}\|_p) \|\bfx_{\sfi}\|_p +F_r^{(l)}(\bfx) + \frac{\beta}{2}\| \bfx - \bfx^{(l)} \|_{2}^{2}
\end{equation}
over group support set $\sfS^{(l)}$. For the brevity of notations, we still use the boldface $\bfx,\bfy,\bfz,\cdots$ to denote the vectors  on $\sfS^{(l)}$ in the following.

Equivalently, we can solve the following constrained optimization problem by
\begin{equation}\label{eq:alg-constrained}
\begin{split}
  &\min_{\bfz} \sum_{ \sfi \in \sfS^{(l)} } \phi^{\prime}(\|\bfx_{\sfi}^{(l)}\|_p) \norm{\bfz_{\sfi}}_p +f_r(\bfs) + \frac{\beta}{2}\| \bfx - \bfx^{(l)} \|_{2}^{2}\\
  &\mbox{s.t. } \bfz=\bfx,~\bfs=\bfA_{\sfS^{(l)}}\bfx-\bfy,
  \end{split}
\end{equation}
where
\begin{equation*}
f_r(\bfs)= \begin{cases}
\frac{1}{r\alpha}\|\bfs\|_r^r,& r\geq 1,\\
\frac{1}{\alpha}\|\bfs\|_{\infty},& r=\infty.
\end{cases}
\end{equation*}

We introduce the penalty parameters $\rho_1,\rho_2>0$ (denoted by $\rho=(\rho_1,\rho_2)$ ) and the Lagrangian multipliers $\vec{\lambda},\vec{\mu}$, then the scaled augmented Lagrangian functional for the weighted problem~\eqref{eq:alg-constrained} at $l$-th step is the following:
\[
\begin{split}
   \mathcal{L}_{\rho}^{(l)}(\bfx,\bfz,\bfs;\vec{\lambda},\vec{\mu}) & = \sum_{ \sfi \in \sfS^{(l)} } \phi^{\prime}(\|\bfx_{\sfi}^{(l)}\|_p) \norm{\bfz_{\sfi}}_p	+f_r(\bfs) + \frac{\rho_1}{2}\left(\norm{\bfA_{\sfS^{(l)}}\bfx-\bfy-\bfs+\vec{\lambda}}_2^2-\norm{\vec{\lambda}}_2^2\right)  \\
   & \phantom{=;} + \frac{\rho_2}{2}\left(\norm{\bfx-\bfz+\vec{\mu}}_2^2-\norm{\vec{\mu}}_2^2\right)+\frac{\beta}{2}\| \bfx - \bfx^{(l)} \|_{2}^{2}.
\end{split}
\]

The scaled ADMM for solving~\eqref{eq:alg-constrained} is described as follows. When there is no confusion with the notations, we use $\bar{\bfx}^{(i)},\bfs^{(i)},\bfz^{(i)}$ to denote the $i$-th iteration step in the inner loop of scaled ADDM.

\vspace{3mm}
\begin{mdframed}[frametitle = {Scaled ADMM: Scaled Alternating Direction Method of Multipliers for Solving~\eqref{eq:alg-constrained}}, frametitlerule = true]

\noindent{\bf Initialization:} Start with $\bar{\bfx}^{(0)} = \bfx^{(l)}_{\sfS^{(l)}}, \vec{\lambda}^{(0)} = \mathbf{0}, \vec{\mu}^{(0)}= \mathbf{0}$.

\noindent {\bf Iteration:} For $i= 0, 1, \ldots, \textmd{MAXit}$,
\begin{enumerate}
\renewcommand{\labelenumi}{ \arabic{enumi}.}
  \item Compute
\begin{equation}\label{eq:admm-s-t}
  (\bfz^{(i+1)}, \bfs^{(i+1)}) = \arg \min_{\bfz, \bfs} \mathcal{L}^{(l)}_{\rho}(\bar\bfx^{(i)},\bfz,\bfs;\vec{\lambda}^{(i)},\vec{\mu}^{(i)}).
\end{equation}
  \item  Compute
\begin{equation}\label{eq:admm-x}
  \bar\bfx^{(i+1)} = \arg \min_{\bar{\bfx}} \mathcal{L}^{(l)}_{\rho}(\bar\bfx,\bfz^{(i+1)},\bfs^{(i+1)};\vec{\lambda}^{(i)},\vec{\mu}^{(i)}).
\end{equation}
  \item Update
\begin{align}
  \boldsymbol\lambda^{(i+1)} & = \boldsymbol\lambda^{(i)} +  \bfA\bar{\bfx}^{(i+1)}-\bfy-\bfs^{(i+1)},\\
  \boldsymbol\mu^{(i+1)} & = \boldsymbol\mu^{(i)} +  \bar{\bfx}^{(i+1)}-\bfz^{(i+1)}.
\end{align}
\end{enumerate}
\end{mdframed}

\vspace{2mm}

\subsection{Solving ~\eqref{eq:admm-s-t} and~\eqref{eq:admm-x}}
The subproblems~\eqref{eq:admm-s-t} and~\eqref{eq:admm-x} can be  efficiently solved.
\begin{enumerate}
\item The minimization subproblem in~\eqref{eq:admm-s-t} is equivalently to solving
\begin{equation*}
\min_{\bfz, \bfs} \sum_{ \sfi \in \sfS^{(l)} } \phi^{\prime}(\norm{\bfx_{\sfi}^{(l)}}_p) \norm{\bfz_{\sfi}}_p	+ f_r(\bfs) + \frac{\rho_1}{2}\norm{\bfA_{\sfS^{(l)}}\bar{\bfx}^{(i)}-\bfy-\bfs+\vec{\lambda}^{(i)}}_2^2
    + \frac{\rho_2}{2}\norm{\bar{\bfx}^{(i)}-\bfz+\vec{\mu}^{(i)}}_2^2,
\end{equation*}
which can be separated into two independent subproblems.
\begin{enumerate}
\item $\bfz$-minimization problem:
\begin{equation*}
\min_{\bfz} \sum_{ \sfi \in \sfS^{(l)} } \phi^{\prime}(\norm{\bfx_{\sfi}^{(l)}}_p) \norm{\bfz_{\sfi}}_p
    + \frac{\rho_2}{2}\norm{\bar{\bfx}^{(i)}-\bfz+\vec{\mu}^{(i)}}_2^2.
\end{equation*}
For $p=1$, we have the explicit solution by~\cite{Yin2008Bregman},
\begin{equation*}
\bfz_{\sfi,j}^{(i+1)}=\sgn{\left(\bar{\bfx}^{(i)}_{\sfi,j}+\vec{\mu}^{(i)}_{\sfi,j}\right)}\cdot \max \left\{\abs{\bar{\bfx}^{(i)}_{\sfi,j}+\vec{\mu}^{(i)}_{\sfi,j}}-w_{\sfi}/\rho_2,0   \right\},~w_{\sfi}= \phi^{\prime}(\|\bfx_{\sfi}^{(l)}\|_1).
\end{equation*}
For $p=2$, this group problem is separable, the minimizer of it can be also explicitly given by the shrinkage lemma in~\cite{Yangj2011,wen2010fast,Wu2010Augmented}:
\begin{equation*}
\bfz_{\sfi}(\bfv_{\sfi})=\max\{ \|\bfv_{\sfi}\|_2-\phi^{\prime}(\|\bfx_{\sfi}^{(l)}\|_2)/\rho_2,0 \}\frac{\bfv_{\sfi}}{\|\bfv_{\sfi}\|_2},~~\bfv_{\sfi}=\bar{\bfx}_{\sfi}^{(i)}+\vec{\mu}_{\sfi}^{(i)}.
\end{equation*}
For the general $p>1$, it is strongly convex, we can use standard nonlinear numerical methods, such as Newton method to solve it.

\item $\bfs$-minimization problem:
\begin{equation*}
\min_{\bfs} f_r(\bfs) + \frac{\rho_1}{2}\norm{\bfA_{\sfS^{(l)}}\bar{\bfx}^{(i)}-\bfy-\bfs+\vec{\lambda}^{(i)}}_2^2.
\end{equation*}
For $r=1$, it is a same problem as $\bfz$-minimization one for $p=1$, we omit it here.

For $r=2$, the solution can be obtained easily,
\begin{equation*}
\bfs_{\sfi,j}=\alpha\rho_1\bfv_{\sfi,j}/(1+\alpha\rho_1), ~~\bfv=\bfA_{\sfS^{(l)}}\bar{\bfx}^{(i)}-\bfy+\vec{\lambda}^{(i)}.
\end{equation*}

For general $r>1$, we also can use the standard nonlinear numerical methods to solve it efficiently.

For $r=\infty$, the $\bfs$-minimization problem reads,
\begin{equation*}
\min_{\bfs} \frac{1}{\alpha}\| \bfs \|_{\infty} +\frac{\rho_1}{2}\norm{\bfs-\bfv}_2^2.
\end{equation*}
Let $\widetilde{\bfs},\widetilde{\bfv}$ are sorted from $\bfs,\bfv$ by the absolute values of elements of the known vector $\bfv$ in ascending order, it is equivalent to solving,
\begin{equation}\label{s-problem:infinity}
\min_{\widetilde{\bfs}} \| \widetilde{\bfs} \|_{\infty} + \beta\norm{\widetilde{\bfs}-\widetilde{\bfv}}_2^2,~~\beta=\alpha\rho_1/2.
\end{equation}
Its optimal solution can be obtained by Theorem \ref{thm:infty} in the next subsection,
\begin{equation*}\label{eq:optimal-solu-infty}
\widetilde{\bfs}^{\ast}=\begin{cases}
\widetilde{\bfv}_i  ,&i<i^{\ast}\\
\sgn(\widetilde{\bfv}_i)t_{i^{\ast}},&i\geq i^{\ast},
\end{cases}
\end{equation*}
where $i^{\ast}\in\{0,1,,2,\cdots,n-1\}$ and $t_{i^{\ast}}$ satisfies~\eqref{eq-ti-star}.
 \end{enumerate}
\item The minimization problem in~\eqref{eq:admm-x} is equivalent to solving
\begin{equation*}
\min_{\bar{\vec{x}}} \frac{\rho_1}{2}\norm{\bfA_{\sfS^{(l)}}\bar{\bfx}-\bfy-\bfs^{(i+1)}+\vec{\lambda}^{(i)}}_2^2 + \frac{\rho_2}{2}\norm{\bar{\bfx}-\bfz^{(i+1)}+\vec{\mu}^{(i)}}_2^2+\frac{\beta}{2}\| \bar{\bfx} - \bfx^{(l)} \|_{2}^{2}.
\end{equation*}
The optimality condition is a linear system like,
\begin{equation*}
(\rho_1\bfA_{\sfS^{(l)}}^T\bfA_{\sfS^{(l)}}+(\rho_2+\beta)\bfI)\bar{\bfx}=\rho_1\bfA_{\sfS^{(l)}}^T(\bfy+\bfs^{(i+1)}-\vec{\lambda}^{(i)})+\rho_2(\bfz^{(i+1)}-\vec{\mu}^{(i)})+\beta \bfx^{(l)}.
\end{equation*}
We can solve it by the inverse of a symmetric positive-definite matrix.
\end{enumerate}

\begin{remark}
In fact, when $r=2$, it is unnecessary to introduce the variable $\bfs$. The scaled ADMM can be simplified in this case.
\end{remark}

\subsection{The analytical solution for the $\bfs$-problem with infinity norm}

Now we consider the equivalent $\bfs$-minimization problem for $r=\infty$ in~\eqref{s-problem:infinity}. It is strongly convex, so it has a unique solution.

\begin{theorem}\label{thm:infty}
Suppose $\widetilde{\bfs},\widetilde{\bfv}\in \mathbb{R}^n$, and the elements of $\widetilde{\bfv}$ is in ascending order by $|\widetilde{\bfv}_1|\leq |\widetilde{\bfv}_2|\cdots\leq|\widetilde{\bfv}_n|$, then the minimization problem
\begin{equation}\label{s-problem:infinity-theorem}
\min_{\widetilde{\bfs}} \| \widetilde{\bfs} \|_{\infty} + \beta\norm{\widetilde{\bfs}-\widetilde{\bfv}}_2^2,
\end{equation}
has the explicit optimal solution,
\begin{equation}\label{eq:potimal-solution-infty}
\widetilde{\bfs}^{\ast}=\begin{cases}
\widetilde{\bfv}_i  ,&i<i^{\ast}\\
\sgn(\widetilde{\bfv}_i)t_{i^{\ast}},&i\geq i^{\ast}.
\end{cases}
\end{equation}
where $i^{\ast}$ is a specific element of $\{0,1,\cdots, n-1\}$ such that
\begin{equation}\label{eq-ti-star}
	t_{i^{\ast}}=\frac{1}{n-{i^{\ast}}}\left(\sum_{j={i^{\ast}}+1}^n |\widetilde{\bfv}_j|-\frac{1}{2\beta}\right)~\mbox{ and } ~t_{i^{\ast}}\in [|\widetilde{\bfv}_{i^{\ast}}|,|\widetilde{\bfv}_{{i^{\ast}}+1}|]
\end{equation}
holds simultaneously.

\end{theorem}

\begin{proof}

	Suppose $s_{\infty}=\|\widetilde{\bfs}\|_{\infty}$. The minimization problem \eqref{s-problem:infinity-theorem} can be rewritten to be more simple,
\begin{equation}\label{vv}
  \min_{s_{\infty}}  f(s_{\infty}) =s_{\infty} + \beta \sum_{|\widetilde{\bfv}_{i} | > s_{\infty}}( |\widetilde{\bfv}_{i}|-s_{\infty}  )^2.
\end{equation}
	We remark here if $s_{\infty}>|\widetilde{\bfv}_n|$, the minimizer is $s_{\infty}=|\widetilde{\bfv}_n|$ when $\widetilde{\bfs}=\widetilde{\bfv}$. This is a contradiction. Hence we can replace $s_{\infty}$ by $0 \leq t \leq |\widetilde{\bfv}_n|$, and the minimization problem (\ref{vv}) can be modified to be
 \[
 \min_{0 \leq t \leq |\widetilde{\bfv}_n|}  f(t) =t + \beta \sum_{|\widetilde{\bfv}_i|  > t} ( |\widetilde{\bfv}_i|-t  )^2.
 \]
 In fact, the objective functional $f(t)$ is a piecewise continuous function. Letting $\widetilde{\bfv}_0=0$, we have
 \[
 f(t) = t + \beta \sum_{j=i+1}^n ( |\widetilde{\bfv}_j|-t  )^2, \quad t \in [|\widetilde{\bfv}_i|,|\widetilde{\bfv}_{i+1}|], \quad i=0,\cdots,n-1.
 \]
 and
\begin{equation}\label{eq:derivative-of-f}
 f'(t) = 1 + 2\beta \sum_{j=i+1}^n ( t-|\widetilde{\bfv}_j| ), \quad t \in (|\widetilde{\bfv}_i|,|\widetilde{\bfv}_{i+1}|),\quad i=0,\cdots,n-1.
 \end{equation}
For $i=1,2,\cdots, n-1$, the right limit of the derivative of $f(t)$ at $t=|\widetilde{\bfv}_i|$ is,
\[
 f'(|\widetilde{\bfv}_{i}|+0)=1 +2\beta(n-i)|\widetilde{\bfv}_{i}|  - 2 \beta \sum_{j=i+1}^n |\widetilde{\bfv}_j|;
 \]
similarly, the left limit of the derivative of $ f(t)$ at $t=|\widetilde{\bfv}_{i}|$ is
 \[
 f'(|\widetilde{\bfv}_{i}|-0)=1 +2\beta(n-i+1) |\widetilde{\bfv}_{i}|  - 2 \beta \sum_{j=i}^n  |\widetilde{\bfv}_j|.
 \]
Since $f(t)$ is continuous at $|\widetilde{\bfv}_{i}|$ and $f'(|\widetilde{\bfv}_{i}|+0)=f'(|\widetilde{\bfv}_{i}|-0)$, $f(t)$ is continuously differentiable.

Furthermore, from~\eqref{eq:derivative-of-f}, we know that the derivative of $f(t)$ is monotonically increasing. Hence $f(t)$ is convex. Thus $f'(t)=0$ can give us the optimal solution of the simplified problem~\eqref{vv}. Let
\[
 t_i= \frac{1}{n-i}(\sum_{j=i+1}^N |\widetilde{\bfv}_j |    -\frac{1}{2 \beta}), ~i=0,1,\cdots,n-1.
\]
If there exists $i^{\ast}$ such that $t_{i^{\ast}} \in [|\widetilde{\bfv}_{i^{\ast}}|,|\widetilde{\bfv}_{i^{\ast}+1}|]$, then $t_{i^{\ast}}$ is the minimizer. Evidently, the optimal solution of minimization~\eqref{s-problem:infinity-theorem} can be given by~\eqref{eq:potimal-solution-infty}.

\end{proof}

\section{Numerical Experiments}\label{sec:experiments}

Numerical experiments are reported in this section to show the efficiency of the InISSAPL algorithm. All of them are implemented on a Laptop (Intel(R) Core(TM) Duo i5-7200u @2.50GHz 2.70GHz, 4.00GB RAM) using Matlab(License ID:1108635).

We consider the numerical tests of application in group sparse signal recovery. Let $\bfx_{or}$ denote the group sparse original signal, which is generated by randomly splitting its components into $\sfg$ groups. For each nonzero group member, its entries are randomly generated as i.i.d. Gaussian. Suppose that $\bfB\in \mathbb{R}^{M\times N}$ is randomly generated by an i.i.d. Gaussian ensemble. We let $\bfA$ be the row orthogonalized matrix of $\bfB$ by $\bfA=(orth(\bfB'))'$ in Matlab code. Then the measurement $\bfy$ is get by
\begin{equation*}
\bfy=\bfA *\bfx_{or}+ \sigma*noise,
\end{equation*}
where $\sigma$ is the noise level and $noise$ represents the three popular ones, Laplace noise, Gaussian noise and uniform noise.

We denote by $s$ the number of nonzero groups of the original signal $\bfx_{or}$. Then the sparsity level $k_s$ is defined by $k_s=s/\sfg$. For simplicity, we consider the uniform group partitions that we have the same group size, denoted by $n$. Define the relative recovery error $\epsilon$ by
\begin{equation*}
\epsilon=\frac{\|\bfx-\bfx_{or}\|_2}{\|\bfx_{or}\|_2}.
\end{equation*}

In our numerical experiments, we set $M=256,N=1024$ for the size of problem, $\sigma=0.001$ for the noise level and $n=8$ for the uniform group size, unless otherwise mentioned. The recovery is recognized as success when the relative error $\epsilon$ is less than $1\%$. For the iteration stopping criteria in the InISSAPL algorithm, we use the same criterion as in~\cite{Boyd2011Distributed} by setting $\epsilon^{\mbox{abs}}=\epsilon^{\mbox{rel}}=10^{-3}$ in the inner scaled ADMM loop, where
\begin{eqnarray*}
&&\|\widehat{\vec{r}}^{(i+1)}\|_2\leq \sqrt{M}\epsilon^{\mbox{abs}}+\epsilon^{\mbox{rel}}\max\left\{ \|\widehat{\bfA}\widehat{\bfx}^{(i+1)}\|_2,\|\widehat{\bfy}\|_2,\|\widehat{\vec{s}}^{(i+1)}\|_2  \right\},\\
&&\|\widehat{\vec{\rho}}\widehat{\bfA}(\widehat{\bfx}^{(i+1)}-\widehat{\bfx}^{(i)})\|_2\leq\sqrt{N}\epsilon^{\mbox{abs}}+\epsilon^{\mbox{rel}}\|\widehat{\vec{\rho}}\widehat{\vec{\lambda}}^{(i+1)}\|_2,
\end{eqnarray*}
with
\[
\widehat{\bfA}=\left[\begin{array}{cc}
\bfA&\\
&\bfI
\end{array}\right],\widehat{\vec{\rho}}=\left[\begin{array}{cc}
\rho_1&\\
&\rho_2
\end{array}\right],
\]

\[
\widehat{\bfy}=\left[\begin{array}{c}
\bfy\\
\vec{0}\end{array}\right],
\widehat{\vec{s}}=\left[\begin{array}{c}
\vec{s}\\
\vec{z}\end{array}\right],
\widehat{\bfx}=\left[\begin{array}{c}
\bfx\\
\bfx\end{array}\right],
\widehat{\vec{\lambda}}=\left[\begin{array}{c}
\vec{\lambda}\\
\vec{\mu}\end{array}\right],
\widehat{\vec{r}}^{(i+1)}=\widehat{\bfA}\widehat{\bfx}^{(i+1)}-\widehat{\bfy}-\widehat{\vec{s}}^{(i+1)}.
\]
We adopt the stopping criterion $\|\bfx^{(l+1)}-\bfx^{(l)}\|_2 / \|\bfx^{(l)}\|_2\leq 10^{-3}$ for the outer iteration. The maximal iteration numbers are set to  MAXit=1000 in the ADMM and MAX=100 in the outer iteration.

\subsection{Experiments on the initialization of the InISSAPL}\label{discuss-init}
We report the results of experiments when the different starting points are chosen in InISSAPL algorithm. The first kind of starting points are $c\mathbb{1}$ with $c\neq 0$. We choose $c=1$ in the test. By setting $p=2, q=0.5, r=2$ for Gaussian noise, we compute the relative errors $\epsilon$. The second kind of starting points are randomly generated as i.i.d. Gaussian. We compute the average relative error $\bar{\epsilon}$ of 1000 different starting points for the same problem setting as in the first kind.

The experiments are performed for different signal recovery problems with three sensing matrices $\bfA_1,\bfA_2,\bfA_3$ and three sparsity cases $s=8,s=16,s=24$. The comparisons are displayed in Table \ref{tab:re-s}.
\begin{table}[tbp]
\centering
\caption{Relative Errors of the reconstruction by InISSAPL with two kinds of starting points.}
	\begin{tabular}{ccccc}
		\toprule
      &   &$\bfA_1$  & $\bfA_2$& $\bfA_3$  \\
      \midrule
$s=8$	&$\epsilon$  & \num{0.0042} &\num{0.0036}  & \num{0.0041} \\
		&$\bar{\epsilon}$ & \num{0.0042} &\num{0.0036}  & \num{0.0041} \\
		\midrule
$s=16	$  &$\epsilon$ & \num{0.0059} &\num{0.0063}  & \num{0.0058}    \\
		 &	$\bar{\epsilon}$ & \num{0.0059} &\num{0.0063}  & \num{0.0058}    \\
		\midrule
$s=24$ &$\epsilon$ & \num{0.4107} &\num{0.0093}  & \num{0.0084}\\
  &	$\bar{\epsilon}$ & \num{0.4013} &\num{0.1016}  & \num{0.0095}    \\
		
 \bottomrule
	\end{tabular}%
\label{tab:re-s}%
\end{table}
It shows that the InISSAPL algorithm is effective and not sensitive to the choice of suggested starting points, even for the less sparsity case $s=24$. Based on this fact, we will choose vector with ones in all elements as starting point in the following experiments.

\vspace{2mm}
The InISSAPL algorithm covers many cases for different choices of $p,q,r$. We discuss them separately in the following subsections.

\subsection{Accessible to diversity of noise}
Our algorithm is applicable to different types of noise. Here we fix $q=1/2,p=2$ and noise level $\sigma=0.01$ to show the performance for three kinds of noise, Laplace noise, Gaussian noise, and uniform distribution noise.

For a specific case of noise, we compare the relative error in Table~\ref{tab:re-r} when the fidelity term uses different $\ell_r~(r=1,2,\infty)$ norms. It is clearly illustrated that $r=1$ is best for Laplace noise, $r=2$ is best for Gaussian noise and $r=\infty$ is best for uniform noise.

\begin{table}[tbp]
\centering
\caption{Relative Error $\epsilon$ over $r$ for the Laplace noise (top), Gaussian noise (middle), uniform noise (bottom) with $p=2,q=0.5,\sigma=0.01$.}
	\begin{tabular}{cccccccc}
		\toprule
 Laplace noise    &  & $\epsilon(r=1)$ &  & $\epsilon(r=2)$ &  & $\epsilon(r=\infty)$\\ \midrule
		$s=4$ &  &  \textbf{0.0370} &  & \num{0.0854} &  &\num{0.0858}     \\
		$s=8$ &  &  \textbf{0.0270} &  & \num{0.0564} &  &\num{0.0588}     \\
		$s=12$ &  & \textbf{0.0362} &  & \num{0.0569} &  &\num{0.0586}     \\
		$s=16$ &  & \textbf{0.0491} &  & \num{0.0635} &  &\num{0.0654}     \\  \midrule
Gaussian noise   &  & $\epsilon(r=1)$ &  & $\epsilon(r=2)$ &  & $\epsilon(r=\infty)$\\ \midrule
		$s=4$ &  & \num{0.0507} &  & \textbf{0.0186} &  &\num{0.0504}     \\
		$s=8$ &  &  \num{0.0440} &  & \textbf{0.0203} &  &\num{0.0405}     \\
		$s=12$ &  & \num{0.0420} &  & \textbf{0.0247} &  &\num{0.0408}     \\
		$s=16$ &  & \num{0.0604} &  & \textbf{0.0297} &  &\num{0.0638}  \\ \midrule
uniform noise   &  & $\epsilon(r=1)$ &  & $\epsilon(r=2)$ &  & $\epsilon(r=\infty)$\\ \midrule
		$s=4$ &  & \num{0.0265} &  & \num{0.0234} &  &\textbf{0.0110}     \\
		$s=8$ &  &  \num{0.0254} &  & \num{0.0216} &  &\textbf{0.0127}     \\
		$s=12$ &  & \num{0.0220} &  & \num{0.0192} &  &\textbf{0.0159}     \\
		$s=16$ &  & \num{0.0178} &  & \num{0.0170} &  &\textbf{0.0147}     \\
      \bottomrule
	\end{tabular}%
\label{tab:re-r}%
\end{table}

\subsection{Choice of $p$ and $q$}

We discuss numerically the InISSAPL algorithm on the parameters $p,q$ in the $\ell_{p,q}$ regularization term. Firstly, letting $p=r=2$, we test the  algorithm when $q$ varies among $\{0.1,0.3,0.5,0.7,0.9\}$. The rate of success on sparsity level is demonstrated in Figure \ref{fig:q-comparison}. It shows that the algorithm performs best when $q=1/2$. This fact is consistent with the numerical results in \cite{HuYaohuaetal2017GSO,Xu2012L12}.

Secondly, we examine the algorithm on commonly used $p=1$ and $p=2$ for the three kinds of noise with $q=1/2$. As suggested in the former Subsection, we use $r=1$ for Laplace noise, $r=2$ for Gaussian noise and $r=\infty$ for uniform noise, respectively. We compare the rate of success on sparsity level in Figure \ref{fig:rate-of-success-p}. It can be observed that the rate of success with $p=1$ is better than it with $p=2$ for Laplace noise and conversely for Gaussian noise. For uniform noise, it has no essential numerical difference between $p=1$ and $p=2$. These results show that different $p$ values may apply to a specific model.

\begin{figure}[htp]
	\centering
	\includegraphics[scale=0.40]{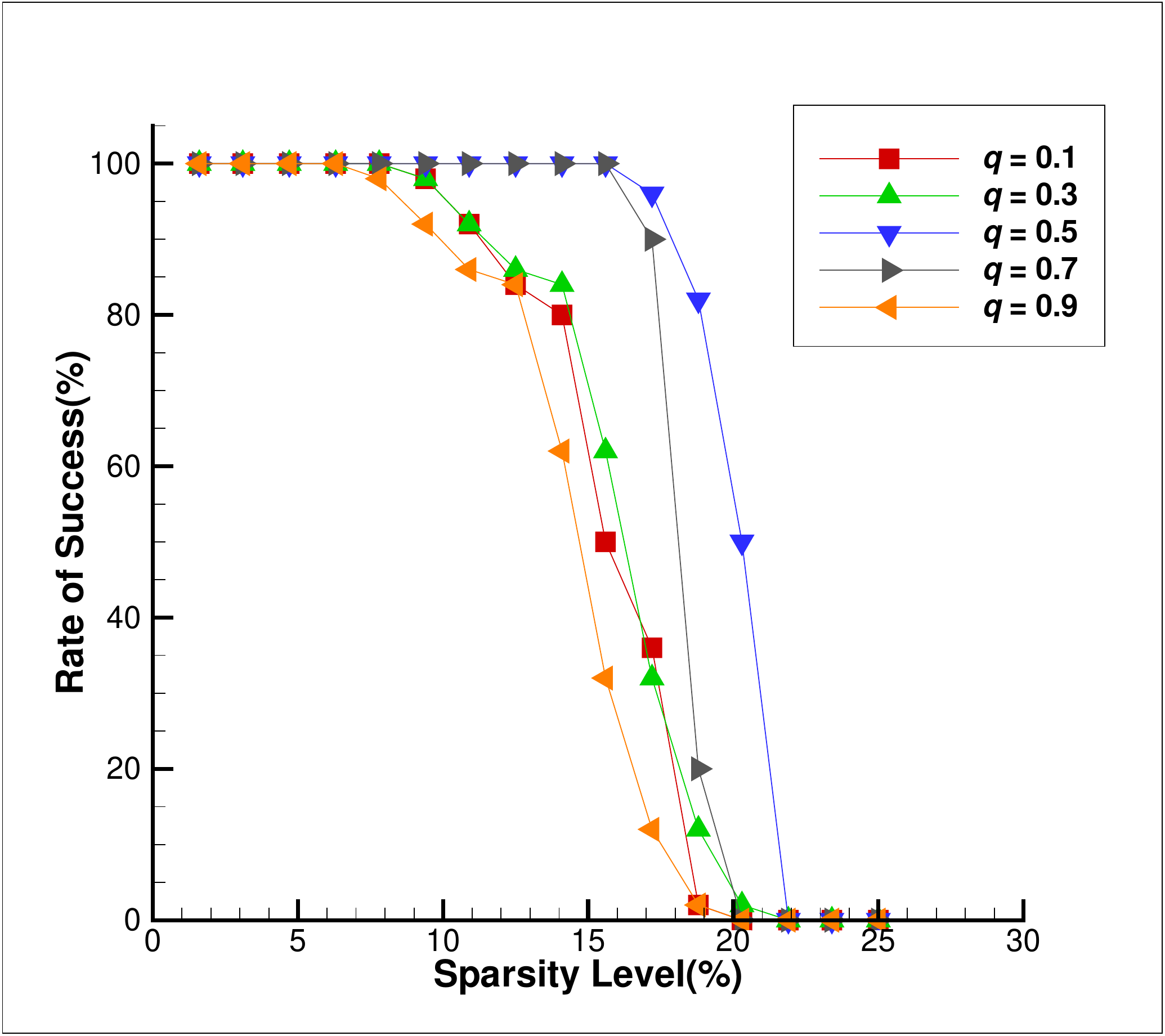}
	\caption{The comparisons on Rate of Success for different $q$ with $p=r=2$.}
	\label{fig:q-comparison}
\end{figure}

\begin{figure}[htp]
	\centering
	\subfloat[Laplace noise($r=1,q=0.5$)]{
   \includegraphics[scale=0.35]{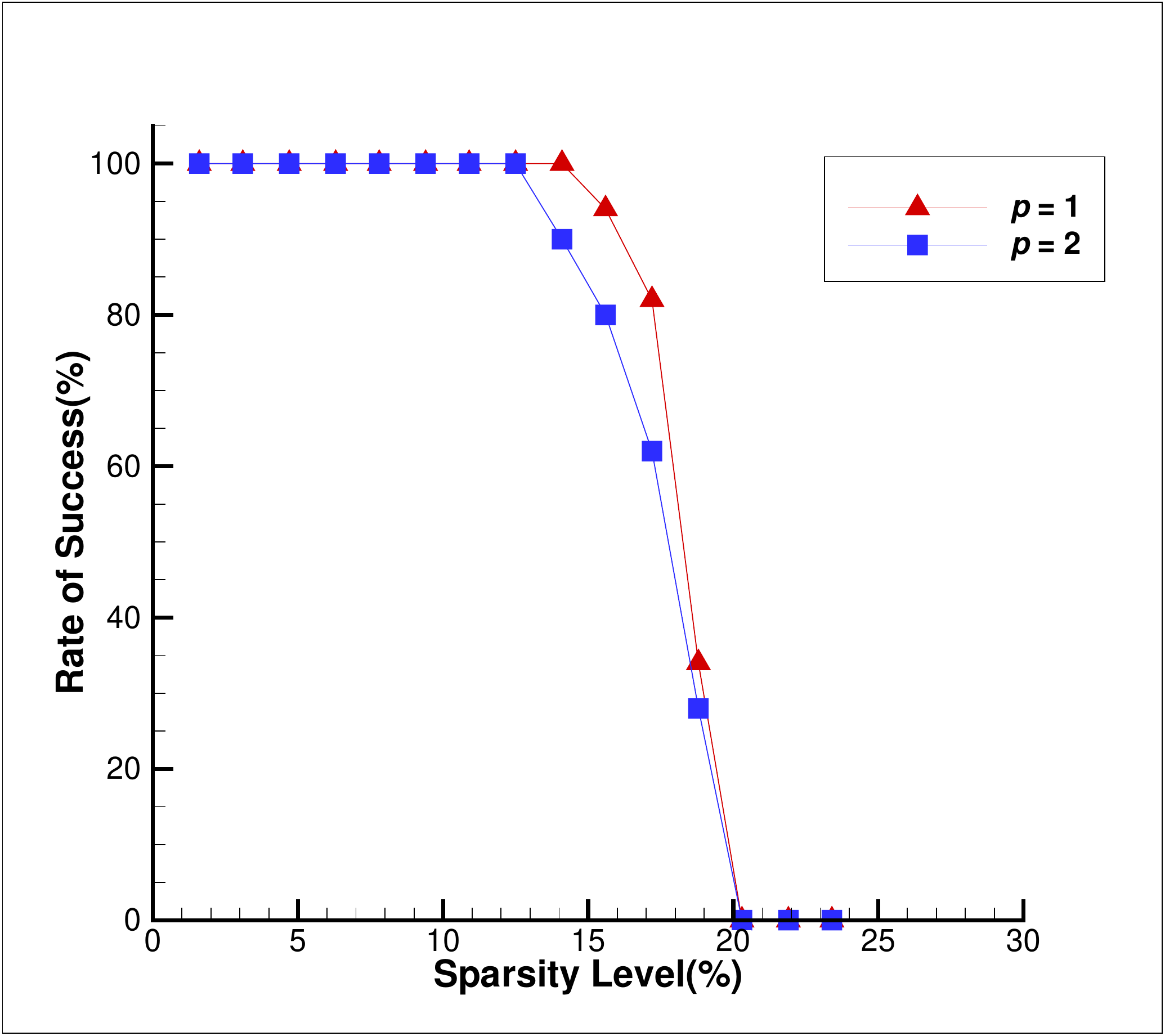}}
\quad
\subfloat[Gaussian noise($r=2,q=0.5$)]{
	\includegraphics[scale=0.35]{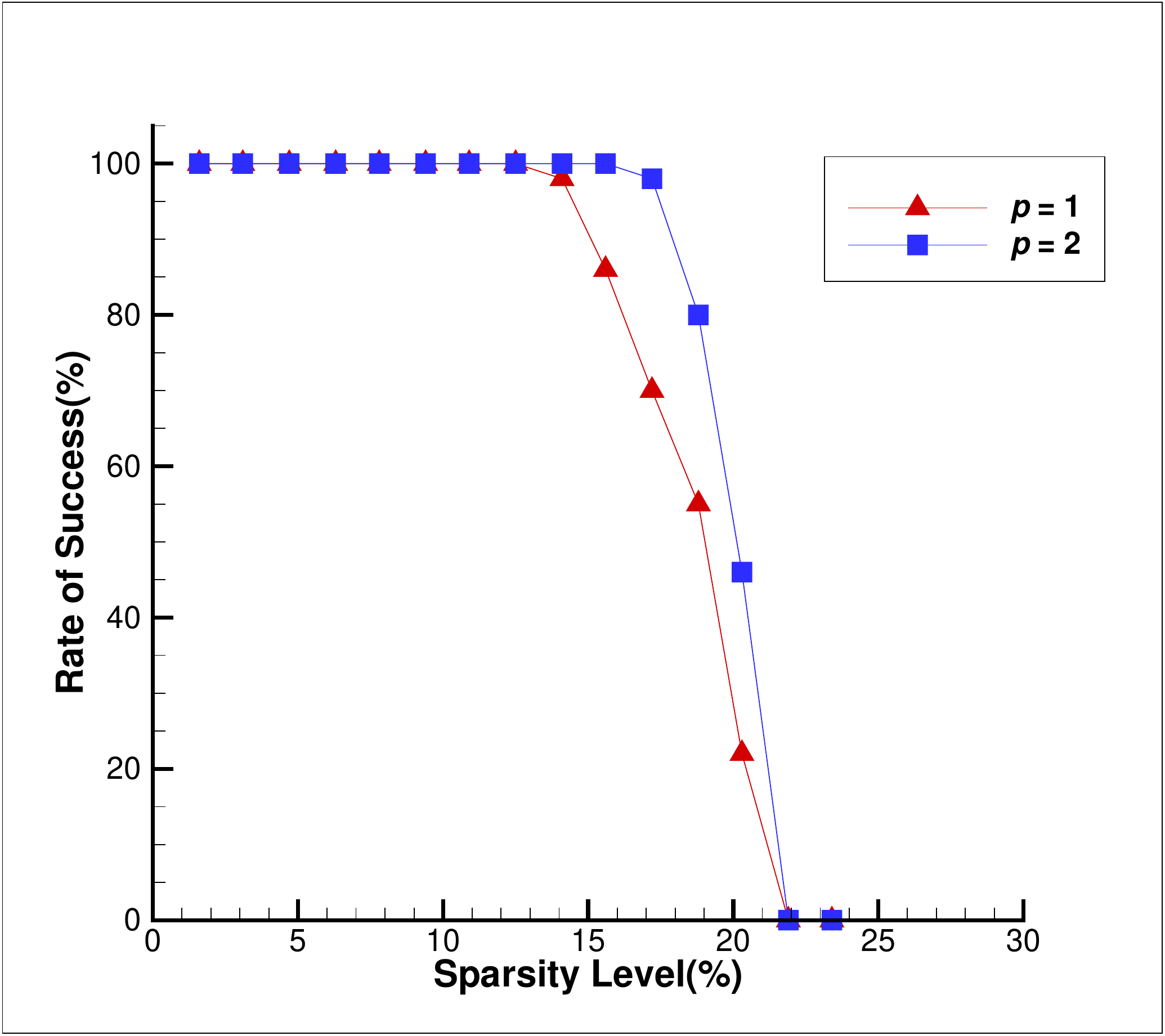}}
\quad
\subfloat[uniform noise($r=\infty,q=0.5$)]{
   \includegraphics[scale=0.35]{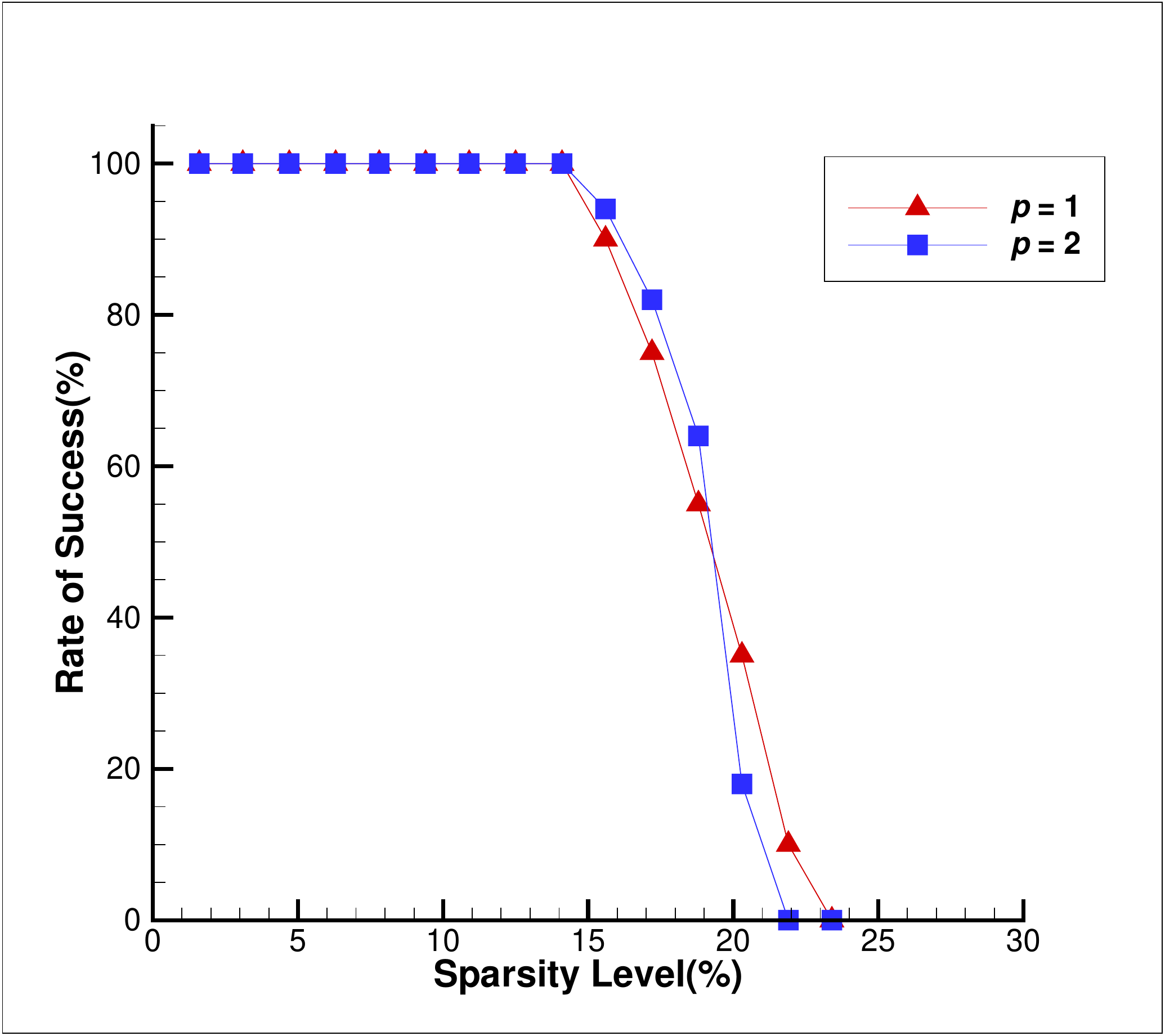}}
	\caption{Comparisons on Rate of Success for Laplace noise (a), Gaussian noise (b) and uniform noise (c) between $p=1$ and $p=2$.}
	\label{fig:rate-of-success-p}
\end{figure}

\subsection{Sensitivity analysis on group size}
In this subsection, we study the sensitivity of our algorithm on group size. We implement the experiments to show the rate of success over the different group sizes ($n=4,8,16,32$) for three types of noise. Similarly as before, we set $r=1,q=1/2$ for Laplace noise, $r=2,q=1/2$ for Gaussian noise and $r=\infty, q=1/2$ for uniform noise. The sensitivity results are given in Figure \ref{fig:sensitivity-analysis} with $p=1$ and $p=2$. It shows that the larger the group size, the higher the rate of success. This fact is true because more information is included for larger group size.

\begin{figure}[htp]
	\centering
	\subfloat[Laplace noise with $p=1$]{
\includegraphics[scale=0.38]{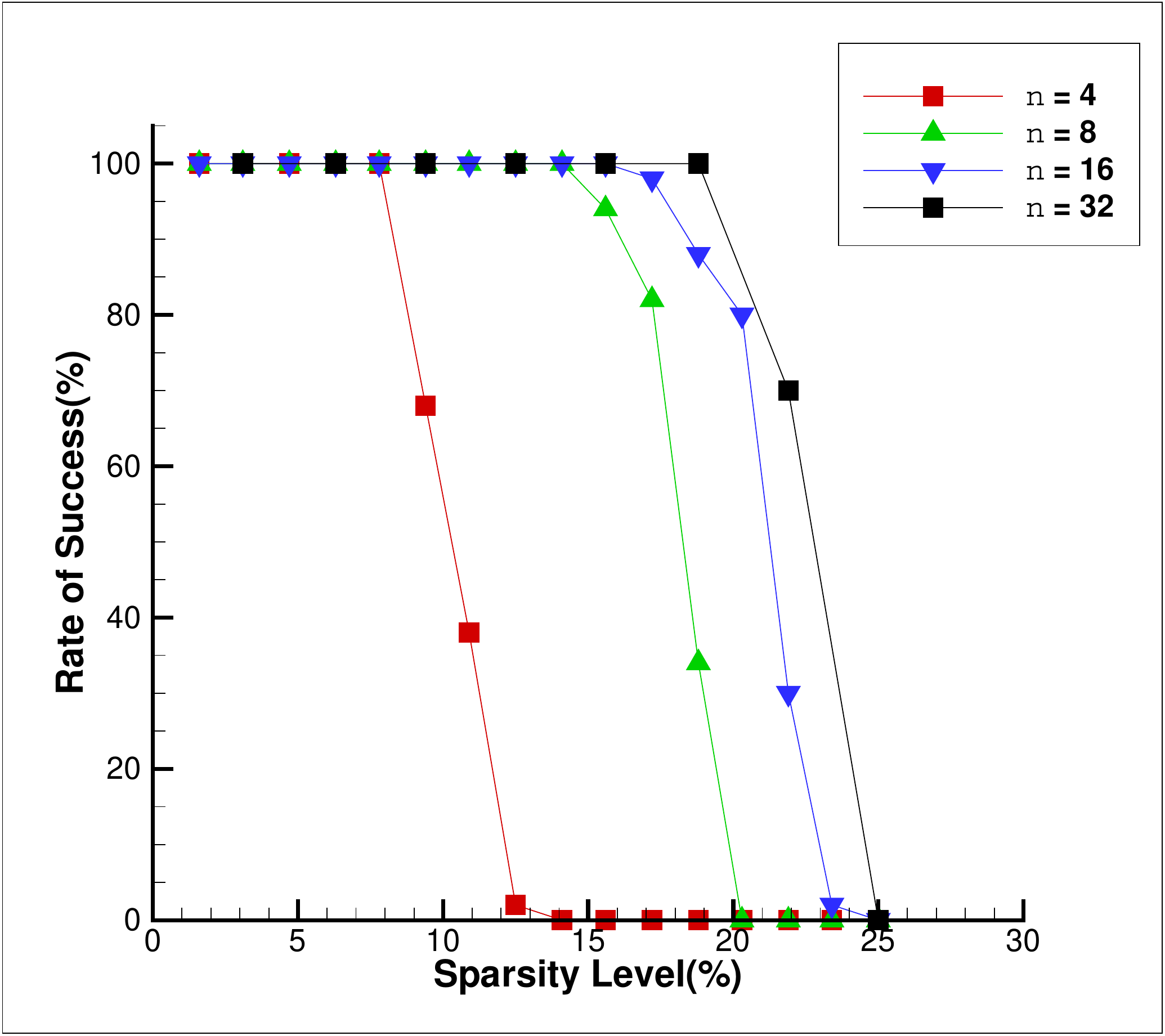}
\label{sens1}
}
\quad
\subfloat[Laplace noise with $p=2$]{
\includegraphics[scale=0.38]{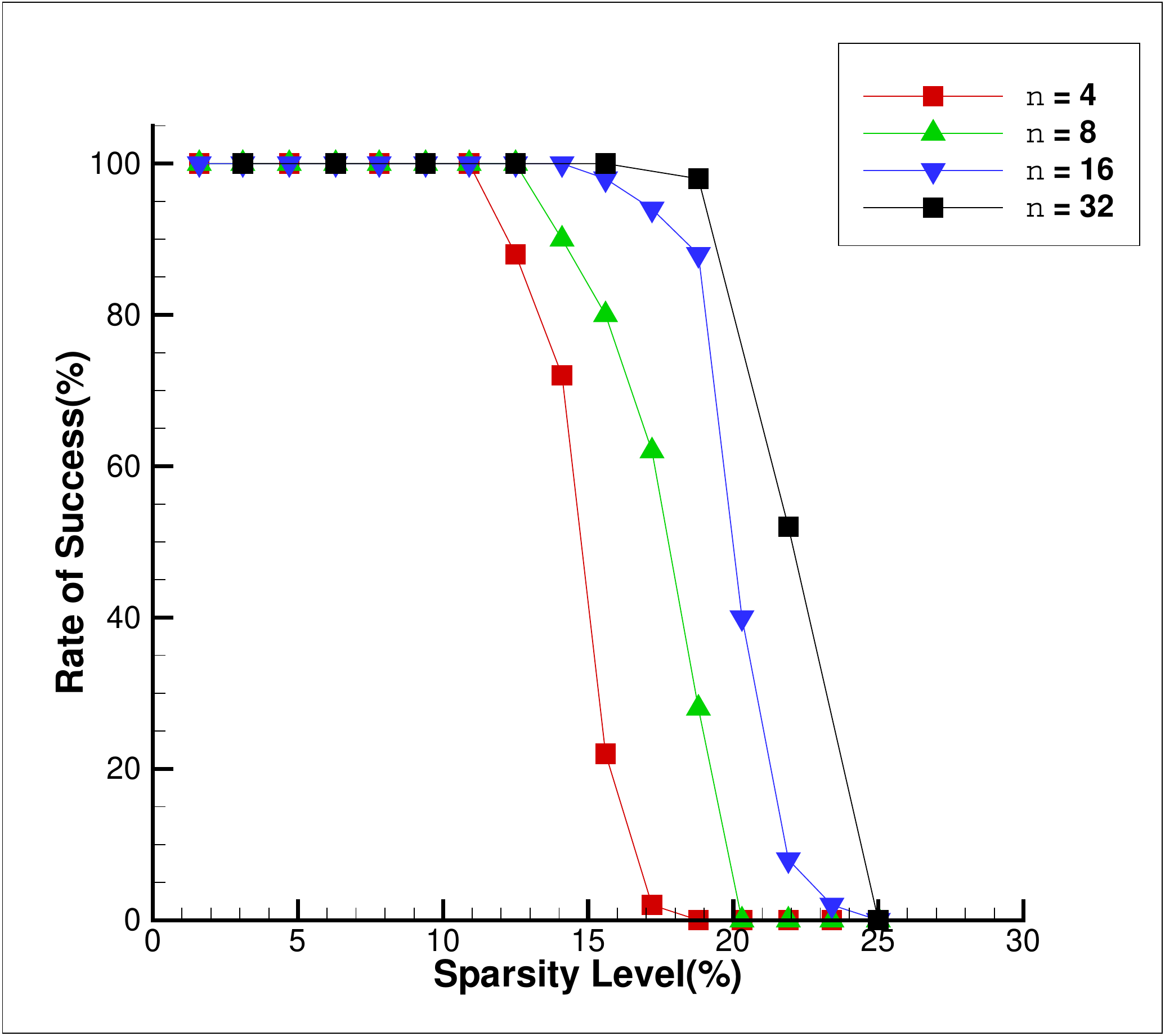}
\label{sens2}
}
\quad
\subfloat[Gaussian noise with $p=1$]{
\includegraphics[scale=0.38]{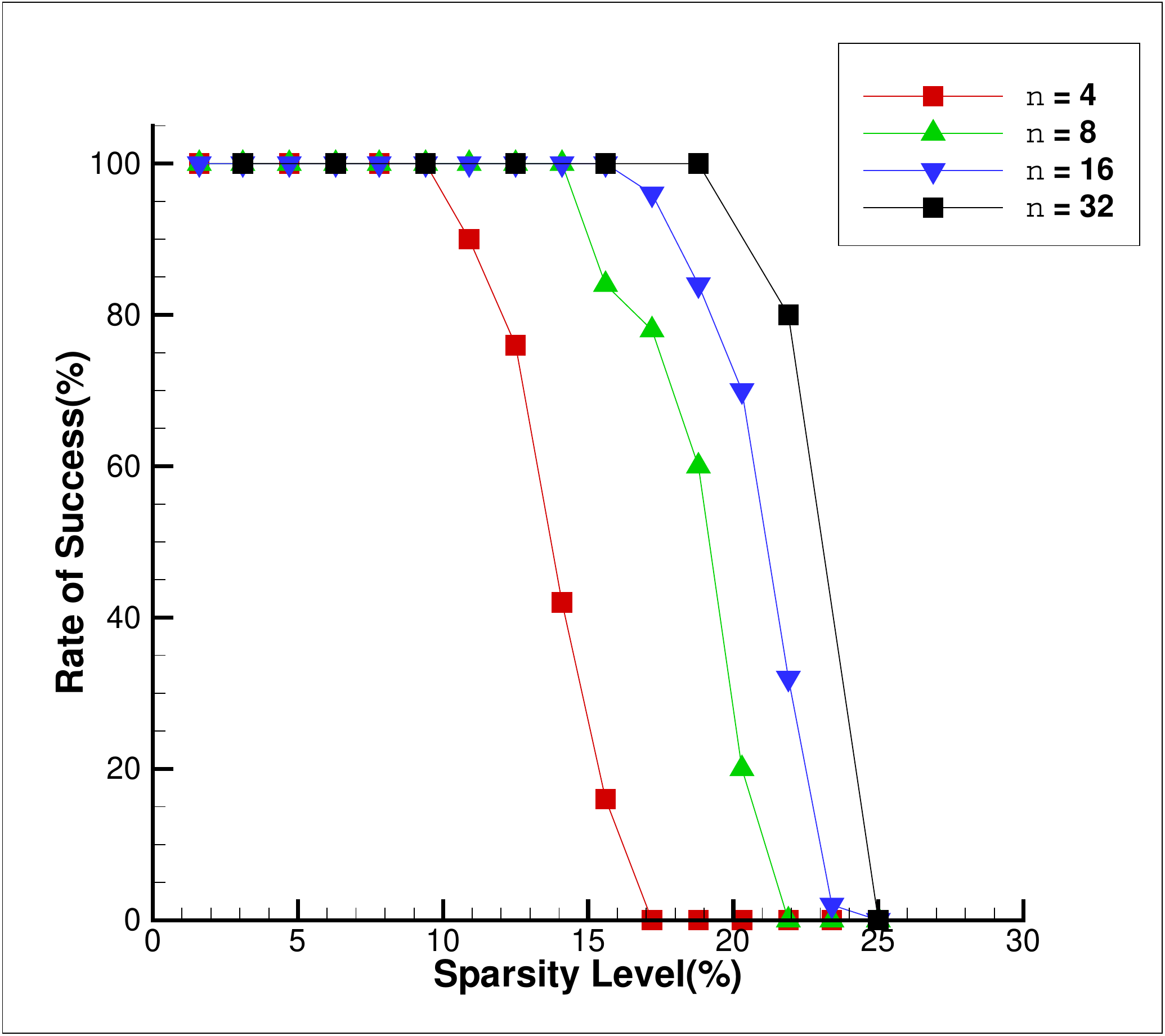}
}
\quad
\subfloat[Gaussian noise with $p=2$]{
\includegraphics[scale=0.38]{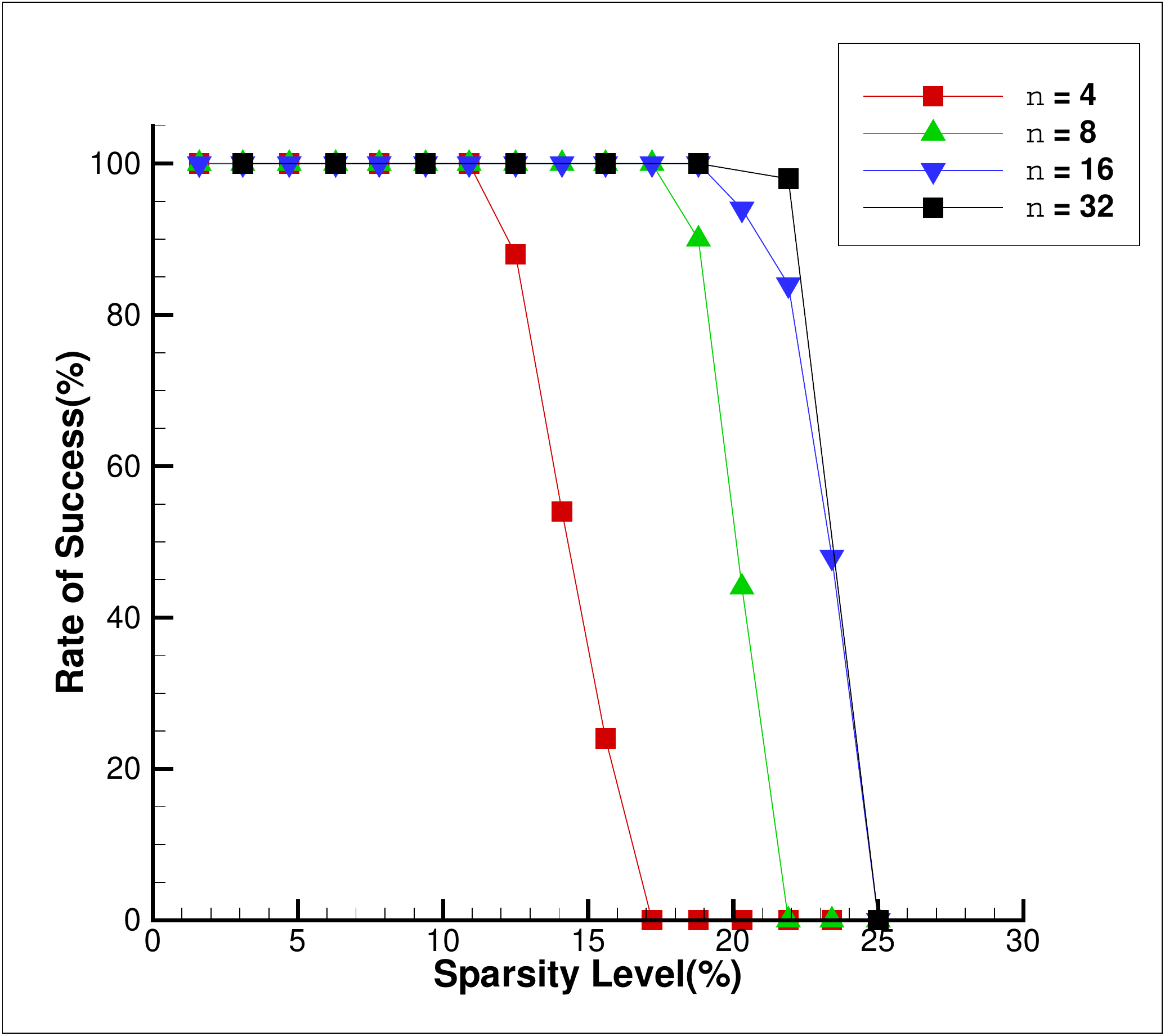}
}
\quad
\subfloat[uniform noise with $p=1$]{
\includegraphics[scale=0.38]{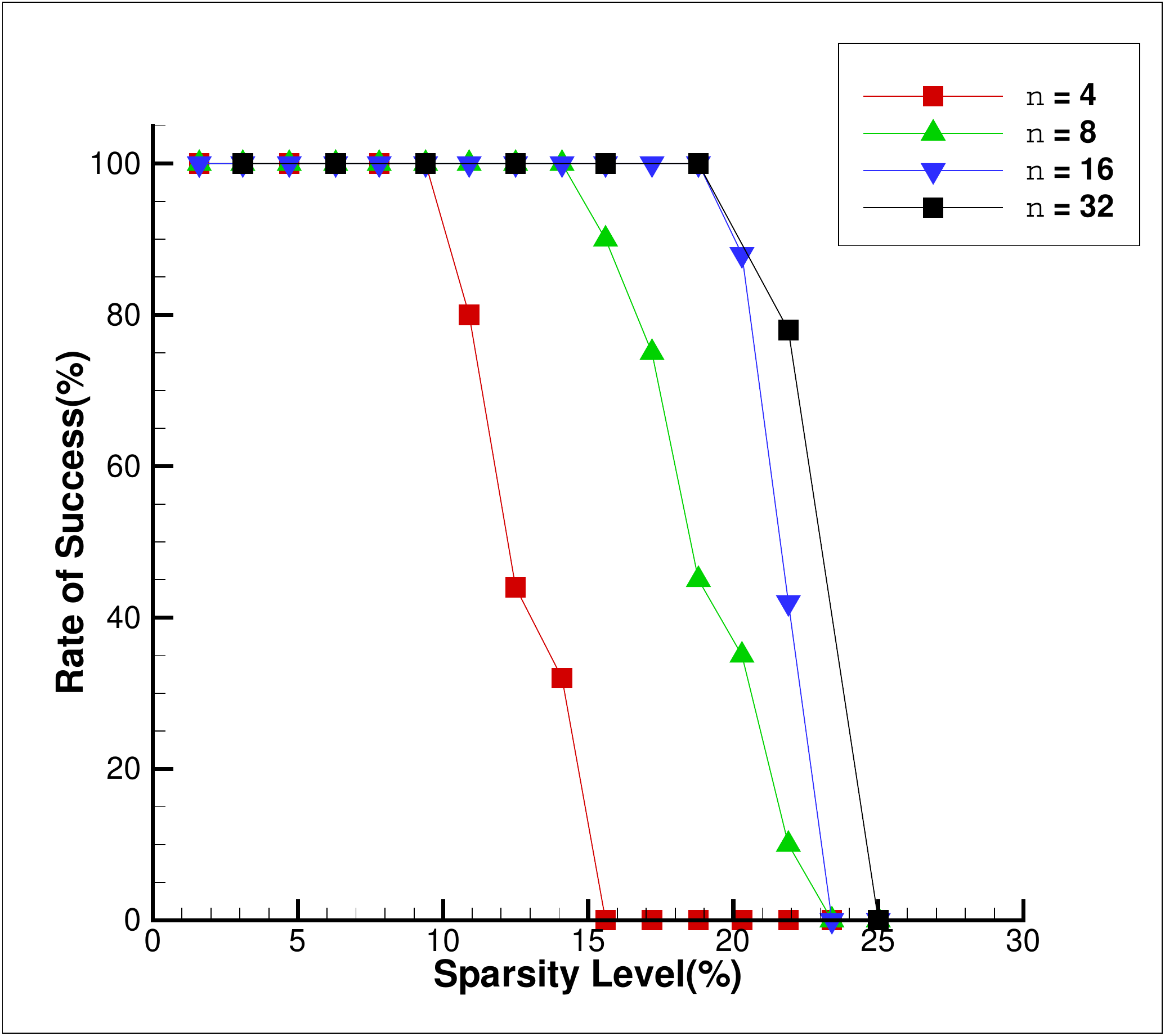}
}
\quad
\subfloat[uniform noise with $p=2$]{
\includegraphics[scale=0.38]{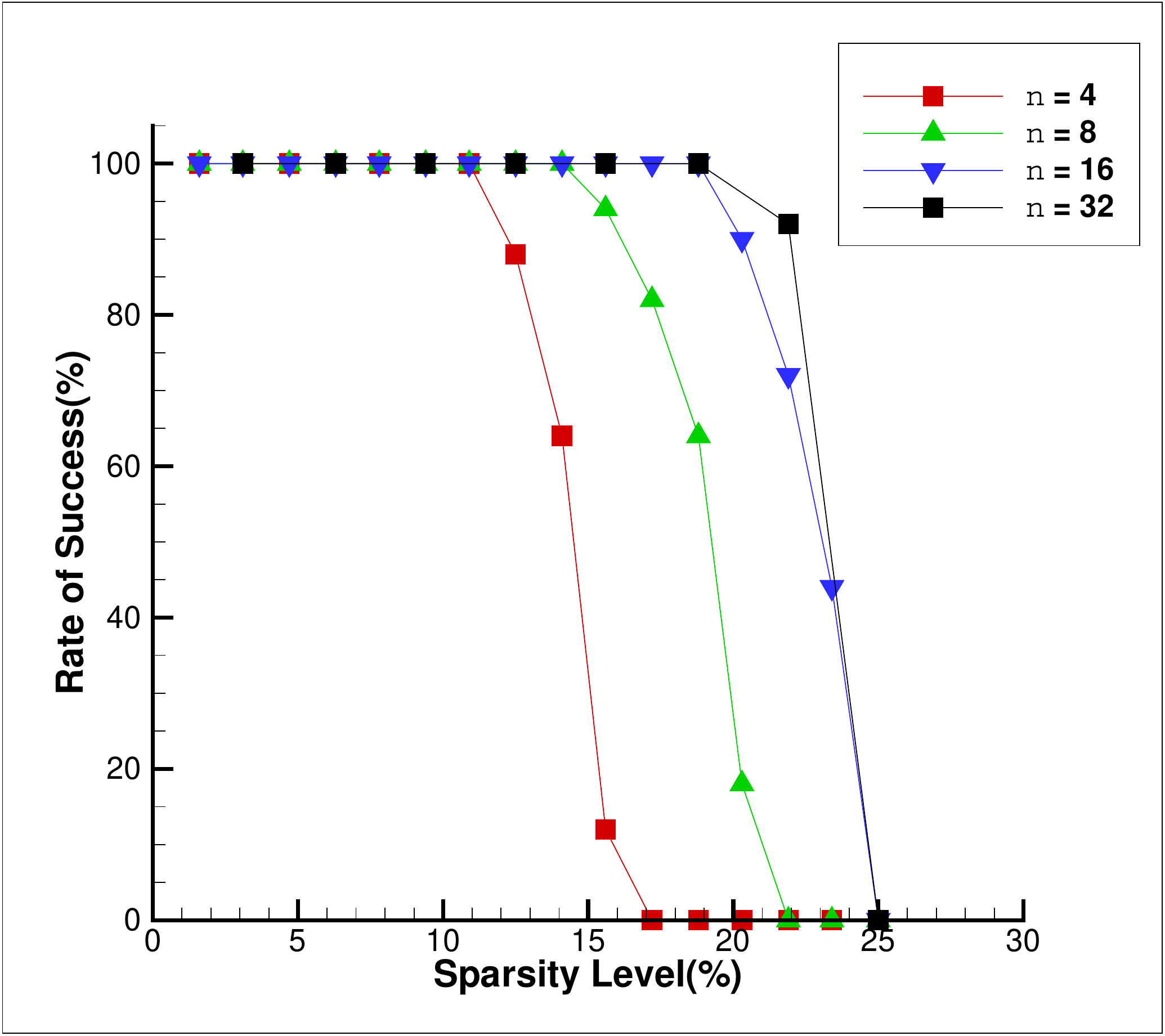}
}
\caption{Sensitivity analysis over group size $n=4,8,16,32$ for Laplace noise (a) and (b), Gaussian noise (c) and (d), uniform noise (e) and (f) with $p=1,p=2$.}
	\label{fig:sensitivity-analysis}
\end{figure}

\subsection{Comparison with some state-of-the-art algorithms }

We compare the InISSAPL algorithm with others in the existing works for the group sparse model. The algorithms are typically PGM-GSO \cite{HuYaohuaetal2017GSO} and the convex optimization Group Lasso \cite{Boyd2011Distributed}. In the code of PGM-GSO algorithm (available online https://CRAN.R-project.org/package=GSparO), there is an additional input: the number of nonzero groups $s$. In our experiments, PGM-GSO denotes their algorithm with EXACT $s$ of the ground truth. Since, in applications, it is hard to know $s$ of the ground truth exactly,  we also use an estimated value $s_e$ (close to the true value $s$) with $s_e=s+2$ in the experiments for more tests. The PGM-GSO with estimated $s_e$ is named e-PGM-GSO. The comparison on rate of success is demonstrated in Figure \ref{fig:comparison} by setting the parameters $p=2,q=1/2,r=2,n=8$ for Gaussian noise. We can see that the rates of success of PGM-GSO (with exact $s$ of the number of nonzero groups of the ground truth) and our InISSAPL are similar, which are considerably higher than e-PGM-GSO and Group Lasso. Note that our InISSAPL does NOT require to input the number of nonzero groups.

For the competitive algorithms, InISSAPL, PGM-GSO, and e-PGM-GSO, we compare the running time and relative error for different sized problems in Table \ref{tab:running-time}. It is illustrated that InISSAPL is more efficient than PGM-GSOers, especially for larger scale problems. The reason is that the computation  is implemented only on the shrinking group support set.

\begin{figure}[htp]
	\centering
	\includegraphics[scale=0.40]{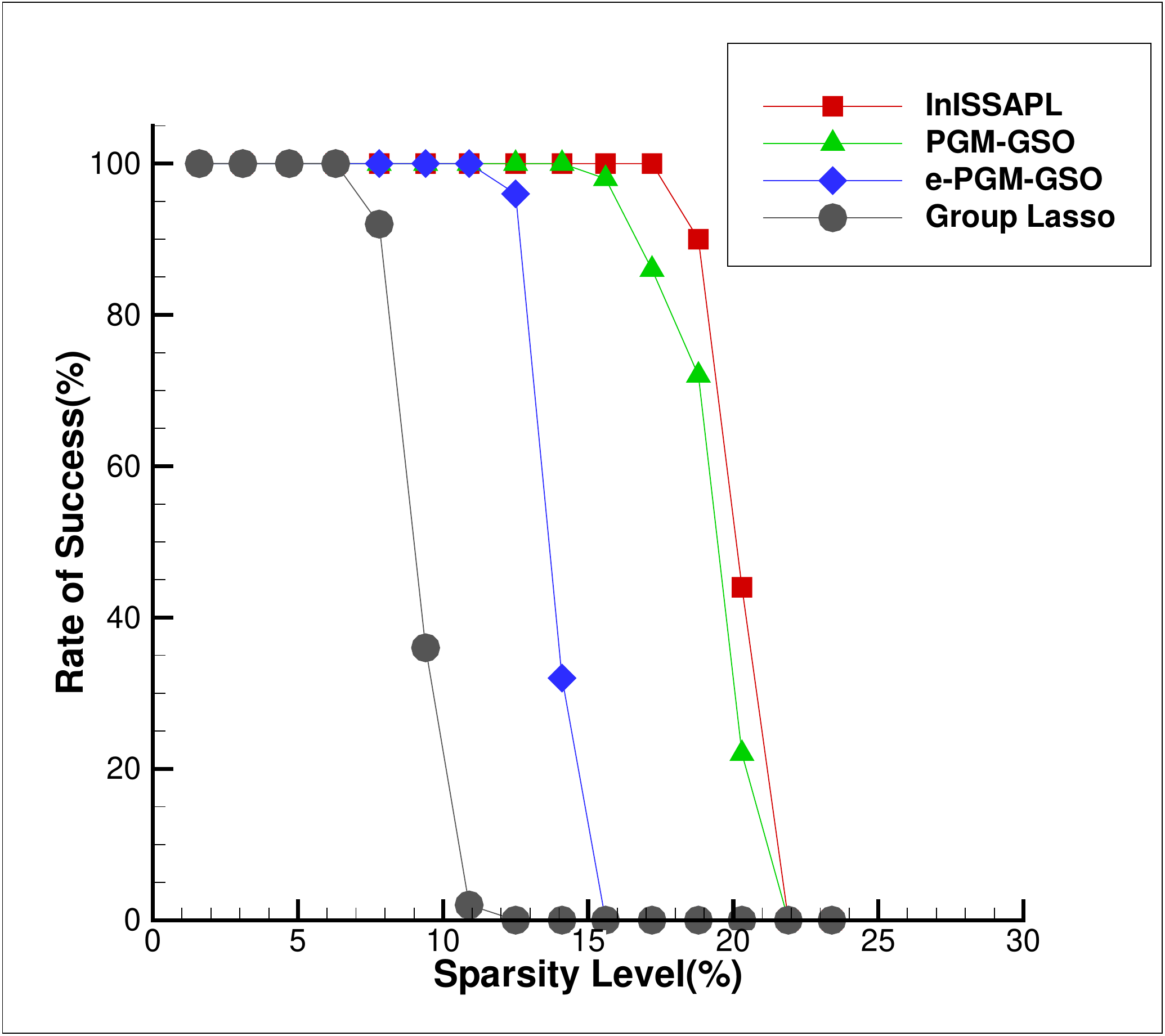}
	\caption{Comparisons on Rate of Success for InISSAPL, PGM-GSO (with true value of the number of nonzero groups $s$), e-PGM-GSO (with estimated value of the number of nonzero groups $s_e=s+2$) and Group Lasso algorithms.}
	\label{fig:comparison}
\end{figure}

\begin{table}[tbp]
	\centering
	\caption{Comparisons on Running time and Relative Error $\epsilon$ for PGM-GSO, e-PGM-GSO, InISSAPL algorithms in two problems with different size. It can be seen that the advantages of our algorithm become larger when the problem scale increases. }
	\begin{tabular}{clllllll}
		\toprule
$M=256$\\$N=1024$ &  & PGM-GSO&  & e-PGM-GSO &  &InISSAPL &   \\ \midrule
		 $s$  &  & Time(s)   &$\epsilon$& Time(s) & $\epsilon$  &Time(s) & $\epsilon$  \\ \midrule
		4 &  &  \num{0.56} &\num{0.0024}  & \num{0.59} &\num{0.0031}  & \textbf{0.46} & \textbf{0.0023}   \\
		8 &  &  \num{0.58} &\textbf{0.0025}  & \num{0.59} &\num{0.0033}  & \textbf{0.49} & \num{0.0027}   \\
		12 &  &  \num{0.58} &\textbf{0.0030}  & \num{0.60} &\num{0.0032} &\textbf{0.50} & \textbf{0.0030}    \\
		16 &  & \num{0.59} &\num{0.0033}  &  \num{0.81} &\num{0.0040}  &\textbf{0.52} & \textbf{0.0031}   \\ \midrule
$M=1024$\\$N=4096$&  & PGM-GSO&  & e-PGM-GSO &  &InISSAPL &   \\ \midrule
		 $s$  &  & Time(s)   &$\epsilon$& Time(s) & $\epsilon$  &Time(s) & $\epsilon$  \\ \midrule
		25 &  &  \num{18.04} &\num{0.0026}  & \num{18.99} &\num{0.0039}  & \textbf{3.98} & \textbf{0.0025}   \\
		50 &  &  \num{18.07} &\textbf{0.0027}  & \num{18.32} &\num{0.0037}  & \textbf{4.20} & \textbf{0.0027}   \\
		75 &  &  \num{18.18} &\num{0.0029}  & \num{18.21} &\num{0.0046}  & \textbf{6.58} & \textbf{0.0028}    \\
		100 &  & \num{18.25} &\num{0.6095}  &  \num{18.87} &\num{0.8928}  & \textbf{9.02} & \textbf{0.0866}   \\  \bottomrule
	\end{tabular}%
	\label{tab:running-time}%
\end{table}

\section{Conclusions}
\label{sec:conclusion}

The group sparse $\ell_{p,q}$-$\ell_r$ model is very useful in many applications. The InISSAPL algorithm provides a unified framework to deal with all the cases of parameters $p\geq 1,0<q<1, 1\leq r\leq \infty$. When proving the global convergence of algorithm with KL property, we develop a lower bound theory for the nonzero groups of the iterative sequence to avoid the non-Lipschitz feature and construct a sophisticated subdifferential formula. Along iterations, the unknowns become fewer and fewer and can be calculated by the scaled ADMM in the inner loop. Therefore it is specially efficient for large-scale problems. Numerical experiments and comparisons demonstrate the good performance of our algorithm.

In our future work, the model and algorithm can be extended to other applications with overlapping groups structure such as the gene expression data and the patch patterns in image processing.

\section{Acknowledgements}
We greatly appreciate helpful discussions with Xue Feng, and thank the authors of \cite{HuYaohuaetal2017GSO} for providing their code available online https://CRAN.R-project.org/package=GSparO. 

\section{Appendix}
\label{sec:appendix}

We firstly recall the basic definitions of subdifferential and horizon cone from the reference~\cite{Rockafellar2009Variational}.

\begin{definition}[Subdifferentials]\label{def-subdiff}
  Let $h: \mathbb{R}^{N} \to \mathbb{R}\cup\{+\infty\}$ be a proper, lower semicontinuous function.
  \begin{enumerate}
    \item The \emph{regular subdifferential} of $h$ at $\bar{\bfx} \in \dom h = \{\bfx \in \mathbb{R}^{N}: h(\bfx) < +\infty  \}$ is defined as
    \[
    \widehat{\partial}h(\bar{\bfx} ) := \set{\bfv \in\mathbb{R}^{N}:
    \liminf_{\substack{ \bfx \to \bar{\bfx} \\ \bfx \neq \bar{\bfx} }}\frac{h(\bfx)-h(\bar{\bfx})- \langle \bfv,\bfx -\bar{\bfx} \rangle}{\|\bfx-\bar{\bfx}\|}\geq 0
    };
    \]
    \item The (limiting) \emph{subdifferential} of $h$ at $\bar{\bfx}  \in \dom h $ is defined as
     \[
    \partial h(\bar{\bfx} ):=\set{\bfv \in\mathbb{R}^{N}: \exists \bfx^{(k)} \to \bar{\bfx} , h(\bfx^{(k)}) \to h(\bar{\bfx}), \bfv^{(k)}\in \widehat{\partial} h(\bfx^{(k)}), \bfv^{(k)} \to \bfv
};
    \]
    \item The \emph{horizon} \emph{subdifferential} of $h$ at $\bar{\bfx}  \in \dom h $ is defined as
     \[
    \partial^{\infty} h(\bar{\bfx} ):=\set{\bfv \in\mathbb{R}^{N}: \exists \bfx^{(k)} \to \bar{\bfx} , h(\bfx^{(k)}) \to h(\bar{\bfx}), \bfv^{(k)}\in \widehat{\partial} h(\bfx^{(k)}), \lambda^{(k)}\bfv^{(k)} \to \bfv \mbox{ for some sequence $\lambda^{(k)}\searrow 0$}
}.
    \]
  \end{enumerate}
\end{definition}

\begin{remark}
From Definition~\ref{def-subdiff}, the following properties hold:
\begin{enumerate}
  \item For any $\bar{\bfx}  \in \dom h $, $\widehat{\partial}h(\bar{\bfx} ) \subseteq \partial h(\bar{\bfx} )$. If $h$ is continuously differentiable at $\bar{\bfx} $, then $\widehat{\partial}h(\bar{\bfx} ) = \partial h(\bar{\bfx} )= \set{\nabla h(\bar{\bfx} )}$;
  \item For any $\bar{\bfx}  \in \dom h $, the subdifferential set $\partial h(\bar{\bfx} )$ is closed, i.e,
\[
\set{\bfv \in\mathbb{R}^{N}:\exists \bfx^{(k)} \to \bar{\bfx}, h(\bfx^{(k)}) \to h(\bar{\bfx}), \bfv^{(k)}\in\partial h(\bfx^{(k)}), \bfv^{(k)} \to \bfv }\subset \partial h(\bar{\bfx} ).
\]
\end{enumerate}
\end{remark}

\begin{definition}[Horizon cone]\label{def-horizon-cone}
For a set $C \subset \mathbb{R}^N $, the \emph{horizon cone} is the closed cone $C^{\infty}$ given by
\[
C^{\infty}=\left\{\begin{array}{ll}
\{\bfv ~|~\exists ~\bfv^{(k)}\in C, \lambda^{(k)}\searrow 0, \lambda^{(k)}\bfv^{(k)}\to \bfv\} &\mbox{when }C\neq \emptyset,\\
\{\bf0\} &\mbox{when } C=\emptyset.
\end{array}\right.
\]
\end{definition}

\begin{remark}
A set $C\subset \mathbb{R}^N$ is bounded if and only if its horizon cone is just the zero cone: $C^{\infty}=\{\bf0\}$.
\end{remark}

Secondly, the Kurdyka-\L ojasiewicz (KL) property~\cite{Lojasiewicz1963Une,Kurdyka1998gradients} is a useful tool for establishing the convergence of bounded sequence. It allows to cover a wide range of problems \cite{Attouch2013Convergence}.

\begin{definition}[Kurdyka-\L ojasiewicz Property]\cite{Attouch2010Proximal} A proper function $h$ is said to have the \emph{Kurdyka-\L ojasiewicz property} at $\bar{\bfx} \in \dom \partial h = \{\bfx \in \mathbb{R}^{\sfN}: \partial h(\bfx) \neq \emptyset\}$ if there exist $\zeta \in (0, +\infty]$, a neighborhood $U$ of $\bar{\bfx}$, and a continuous concave function $\varphi: [0, \zeta) \to \mathbb{R}_{+}$ such that
    \begin{enumerate}
      \item $\varphi(0) = 0$;
      \item $\varphi(0)$ is $C^{1}$ on $(0,\zeta)$;
      \item for all $s \in (0,\zeta)$, $\varphi^{\prime}(s) > 0$;
      \item for all $\bfx \in U$ satisfying $h(\bar{\bfx}) < h(\bfx) < h(\bar{\bfx}) + \zeta$, the Kurdyka-\L ojasiewicz inequality holds:
      \[
        \varphi^{\prime}(h(\bfx) - h(\bar{\bfx})) \dist (0,\partial h(\bfx)) \geq 1.
      \]
      where $\dist (0,\partial h(\bfx)) = \min\{\| \bfv \|: \bfv \in \partial h(\bfx) \}$,
    \end{enumerate}
\end{definition}

A proper, lower semicontinuous function $h$ satisfying the KL property at all points in $\dom \partial h$ is called a \emph{KL function}. One can refer to~\cite{Attouch2013Convergence,Bolte2014Proximal} for examples of KL functions and the application of KL property in optimization theory.

Recently, the KL property has been extended to the definable functions in an o-minimal structure for the nonsmooth version, see~\cite{Kurdyka1998gradients, Dries1996geometric, Attouch2010Proximal, Bolte2007clarke} and the reference therein. The following definitions and theorem are based on them.

\begin{definition}\cite{Attouch2010Proximal}
Let $\mathcal{O}=\{ {\mathcal{O}}_{n} \}_{n\in\mathbb{N}}$ be such that each ${\mathcal{O}}_n$ is a collection of subsets of $\mathbb{R}^n$. The family $\mathcal{O}$ is an {\em o-minimal structure} over $\mathbb{R}$, if it satisfies the following axioms:
\begin{enumerate}
\item Each ${\mathcal{O}}_n$ is a boolean algebra. Namely $\emptyset\in {\cal{O}}_n$ and for each $A,B\in {\mathcal{O}}_n$, $A\cup B, A\cap B$, and $\mathbb{R}^n\backslash A$ belong to ${\mathcal{O}}_n$.
\item For all $A\in {\mathcal{O}}_n$, $A\times \mathbb{R}$ and $\mathbb{R}\times A$ belong to ${\mathcal{O}}_{n+1}$.
\item For all  $A\in {\mathcal{O}}_{n+1}$, $\prod(A):=\{(x_1,\cdots,x_n)\in \mathbb{R}^n| (x_1,\cdots,x_n,x_{n+1})\in A\}$ belongs to $\mathcal{O}_n$.
\item For all $i\neq j$ in $\{1,2,\cdots,n\}$, $\{ (x_1,\cdots,x_n)\in \mathbb{R}^n|x_i= x_j   \}$ belong to $\mathcal{O}_n$.
\item The set $\{ (x_1,x_2)\in \mathbb{R}^2| x_1<x_2   \}$ belongs to $\mathcal{O}_2$.
\item The elements of $\mathcal{O}_1$ are exactly finite unions of intervals.
\end{enumerate}
\end{definition}

\begin{definition}\label{def-definable}\cite{Attouch2010Proximal}
Given an o-minimal structure $\mathcal{O}$ over $\mathbb{R}$. A set $C$ is said to be {\em definable} (in $\mathcal{O}$) if $C$ belongs to $\mathcal{O}$. A function $f:\mathbb{R}^n\to \mathbb{R}\cup \{+\infty\}$ is said to be {\em definable} in $\mathcal{O}$ if its graph belongs to ${\mathcal{O}}_{n+1}$.
\end{definition}

Then the definable function has the following property:
\begin{itemize}
\item finite sums of definable functions are definable;
\item compositions of definable functions are definable;
\item function of $f(y)=\sup_{x\in C} g(x,y)$ is definable if $g(x,y)$ and the set $C$ are definable.
\end{itemize}
As an example~\cite{Dries1996geometric,Attouch2010Proximal}, there exists an o-minimal structure containing the graph of $x^r: \mathbb{R}\to \mathbb{R}, r\in\mathbb{R}$, which is given by
\begin{equation}\label{eq:powerfunction}
a\mapsto \begin{cases}
a^r,&a>0\\
0,&a\leq 0.
\end{cases}
\end{equation}

\begin{theorem}\label{K-L-theorem}\cite{Attouch2010Proximal} Any proper lower semicontinuous function $f: \mathbb{R}^n\to \mathbb{R}\cup\{+\infty\}$ that is definable in an o-minimal structure $\mathcal{O}$ has the Kurdyka-\L ojasiewicz property at each point of $\dom\partial f$.
\end{theorem}

From this theorem and Definition \ref{def-definable}, the objective function $\mathcal{E}$ in this paper is the compositions of definable functions. So it satisfies the KL property.

The following theorem gives a general and important theoretical framework for the convergence of sequence. It has extensive applications recently~\cite{Attouch2013Convergence,Bolte2014Proximal}.

\begin{theorem}\label{thm-convergence} \cite{Attouch2013Convergence,Bolte2014Proximal} Let $f: \mathbb{R}^n\to \mathbb{R}\cup\{+\infty\}$ be a proper lower semicontinous function. Consider a sequence $\{ \bfx^{(l)} \}$ that satisfies

(H1). (Sufficient decrease condition). For each $l$,
\[
f(\bfx^{(l+1)})+a\|\bfx^{(l+1)}-\bfx^{(l)}\|^2\leq f(\bfx^{(l)});
\]

(H2). (Relative error condition). For each $l$, there exists $w^{(l+1)}\in\partial f(\bfx^{(l+1)})$ such that
\[
\|w^{(l+1)}\|\leq b\|\bfx^{(l+1)}-\bfx^{(l)}\|;
\]

(H3). (Continuity condition). There exists a subsequence $\{ \bfx^{(k_l)} \}$ and $\widetilde{\bfx}$ such that
\[
\bfx^{(k_l)} \to \widetilde{\bfx} \mbox{ and } f( \bfx^{(k_l)} )\to f(\widetilde{\bfx}),  \mbox{~ as }j\to \infty.
\]

If $f$ has the KL property at the cluster point $\widetilde{\bfx}$ specified in (H3), then the sequence $\{ \bfx^{(l)} \}$ converges to $\bar{\bfx}=\tilde{\bfx}$ as $l\to \infty$ and $\bar{\bfx}$ is a critical point.

\end{theorem}


\bibliography{lpq_lr_group_ref}

\end{document}